\documentclass[11pt,a4paper]{amsart}

\usepackage{hyperref}
\usepackage{amsmath,amssymb,latexsym,amsthm}
\usepackage{url}
\usepackage{graphicx,color}
\usepackage[notref,notcite,final]{showkeys}
\usepackage{mathdots}

\usepackage[ruled,linesnumbered]{algorithm2e}
\usepackage[margin=1.2in]{geometry}
\usepackage{caption}
\usepackage{mathtools}
\usepackage{soul}

\usepackage{cite}
\usepackage{yhmath}

\newtheorem{thm}{Theorem}
\newtheorem{lemma}[thm]{Lemma}
\newtheorem{prop}[thm]{Proposition}

\newtheorem{remark}[thm]{Remark}

\DeclareMathOperator*{\dif}{d\!}

\graphicspath{{figures/}}

\title[Linearized Boundary Control Method for Damping Reconstruction]{Linearized Boundary Control Method for Damping Reconstruction in an Acoustic Inverse Boundary Value Problem}

\author{Tianyu Yang}
\address{Department of Computational Mathematics Science and Engineering, Michigan State University, East Lansing, MI 48824, USA}
\email{yangti27@msu.edu}
\author{Yang Yang}
\address{Department of Computational Mathematics Science and Engineering, Michigan State University, East Lansing, MI 48824, USA}
\email{yangy5@msu.edu}

\date{}

\begin{document}


\begin{abstract}
We develop a linearized boundary control method for the inverse boundary value problem of determining the damping coefficient in the damped wave equation. The objective is to reconstruct an unknown perturbation in a known background damping from the linearized Neumann-to-Dirichlet map. 

When the linearization is at a constant background damping, we derive a reconstructive algorithm with stability estimates based on the boundary control method in dimension $n\geq 1$. The reconstruction algorithm is implemented in one dimension to validate its numerical feasibility.
When the linearization is at a non-constant background damping, we establish an increasing stability estimate in the time domain in dimension $n\geq 3$.
\end{abstract}

\maketitle

\section{Introduction}

\textbf{Inverse Boundary Value Problem}
Let $\Omega$ be a bounded domain in $\mathbb{R}^n$ ($n\geq 1$) with smooth boundary. Consider the following boundary value problem for the wave equation
\begin{equation}\label{eq:wave}
\left\{
\begin{alignedat}{2}
    \square_{\rho,\sigma} u(t,x)&=0\quad&&\text{ in }(0,2T)\times\Omega,\\
    u(0,x)=\partial_tu(0,x)&=0\quad&&\text{ on } \Omega,\\
    \partial_\nu u(t,x)&=f\quad&&\text{ on } (0,2T)\times\partial\Omega,
\end{alignedat}
\right.
\end{equation}
where $\square_{\rho,\sigma}$ is a linear wave operator defined as
\begin{equation}
\square_{\rho,\sigma}u(t,x)\coloneqq\rho(x)\partial_t^2u(t,x)+\sigma(x)\partial_tu(t,x)-\Delta u(t,x)
\end{equation}
and $\nu$ is the unit outer normal vector field on $\partial\Omega$.
We will refer to the spatially-varying function $\sigma = \sigma(x)\geq 0$ as the \textit{damping coefficient} or simply \textit{damping}, and the function $\rho=\rho(x)>0$ as the \textit{density}.
Throughout the paper, we will assume $\rho=\rho_0$ where $\rho_0\in C^\infty(\overline{\Omega})$ is a known, strictly positive function.

For $f\in C_c^\infty((0,2T)\times\partial\Omega)$ (where $C_c^\infty((0,2T)\times\partial\Omega)$ denotes the set of compactly supported smooth functions on $(0,2T)\times\partial\Omega)$) and $\rho,\sigma\in C^\infty(\overline{\Omega})$, the standard well-posedness theory for the wave equation ensures the existence of a smooth solution $u=u^f(t,x)$ to the boundary value problem~\eqref{eq:wave}.
Therefore, the following Neumann-to-Dirichlet map (ND map) $\Lambda_{\sigma}$ is well defined:
\begin{equation} \label{eq:NDmap}
    \Lambda_{\sigma}: f\mapsto u^f|_{(0,2T)\times\partial\Omega}
\end{equation}
(see Appendix A for the existence of $u$ and boundedness of $\Lambda_\sigma$ for small $\sigma$). In this paper, we are interested in the \textit{inverse boundary value problem (IBVP)} concerning recovery of $\sigma\in C^\infty(\overline{\Omega})$ from $\Lambda_\sigma$.

\medskip
\textbf{Formal Linearization.} In this paper, we are interested in the linearized version of the inverse problem to recover $\sigma\in C^\infty(\overline{\Omega})$ from $\Lambda_\sigma$. To derive the linearized IBVP, we write
\[\sigma=\sigma_0+\varepsilon\dot{\sigma},\qquad u=u_0+\varepsilon\dot{u},\qquad\Lambda_{\sigma}=\Lambda_{\sigma_0}+\varepsilon\dot{\Lambda}_{\dot{\sigma}},\]
where $\sigma_0$ is the known background damping coefficient, $u_0$ is the background solution, $\dot{\sigma}\in C^\infty(\overline{\Omega})$ is a smooth perturbation to $\sigma_0$, and $\dot{u}$ is a smooth perturbation to $u_0$. The asymptotic expansion in $\varepsilon$ gives
\begin{equation}\label{eq:wave0}
\left\{
\begin{alignedat}{2}
    \square_{\rho_0,\sigma_0} u_0(t,x)&=0\quad&&\text{ in }(0,2T)\times\Omega,\\
    u_0(0,x)=\partial_tu_0(0,x)&=0\quad&&\text{ on } \Omega,\\
    \partial_\nu u_0(t,x)&=f\quad&&\text{ on }(0,2T)\times\partial\Omega,
\end{alignedat}
\right.
\end{equation}
and
\begin{equation}\label{eq:wavedot}
\left\{
\begin{alignedat}{2}
    \square_{\rho_0,\sigma_0} \dot{u}(t,x)&=-\dot{\sigma}(x)\partial_tu_0(t,x)\quad&&\text{ in }(0,2T)\times\Omega,\\
    \dot{u}(0,x)=\partial_t\dot{u}(0,x)&=0\quad&&\text{ on } \Omega,\\
    \partial_\nu \dot{u}(t,x)&=0\quad&&\text{ on }(0,2T)\times\partial\Omega,
\end{alignedat}
\right.
\end{equation}
Using these equations, we define the linearized ND map $\dot{\Lambda}_{\dot{\sigma}}$ by
\begin{equation}
    \dot{\Lambda}_{\dot{\sigma}}: f\mapsto\dot{u}^f|_{(0,2T)\times\partial\Omega},
\end{equation}
where $\dot{u}^f$ denotes the solution of~\eqref{eq:wavedot} with the background solution $u_0^f$ satisfying~\eqref{eq:wave0}. 
For ease of notation, we often write $\dot{\Lambda}_{\dot\sigma}$ and $\Lambda_\sigma$ as $\dot{\Lambda}$ and $\Lambda$, respectively.

\medskip
\textbf{Literature:} Inverse boundary value problems (IBVP) for wave equations have been extensively studied in the mathematical literature, with a wide range of results on uniqueness, stability, and reconstruction methods. When $\sigma \equiv 0$, the density $\rho(x)$ is related to the wave speed $c(x)$ by $\rho(x) = c^{-2}(x)$. Belishev~\cite{belishev1988approach} proved that the wave speed $c$ (and hence $\rho$) is uniquely determined from boundary measurements using the boundary control (BC) method combined with Tataru’s unique continuation theorem~\cite{tataru1995unique, Tataru99unique}. The BC method was subsequently generalized to the reconstruction of Riemannian manifolds~\cite{belishev1992reconstruction}, symmetric time-independent lower-order perturbations of wave operators~\cite{MR1331288}, as well as non-symmetric, time-dependent, and matrix-valued lower-order perturbations~\cite{MR1784415, MR3880231}.
When $\sigma \not\equiv 0$, Pestov~\cite{pestov2012inverse} extended the BC method to include damping and proved the unique determination of the damping coefficient under the assumption that the density is known. 
This result was further extended in the subsequent work~\cite{pestov2014determining} by showing that the damping coefficient and density can be determined simultaneously. Overall, the BC method has proven to be a powerful framework for establishing uniqueness in a broad class of inverse problems. We refer to~\cite{MR2353313, MR1889089, belishev2017boundary} for comprehensive surveys of the BC method.

Stability estimates for the IBVP have been derived using both microlocal analysis and the BC method. Under appropriate geometric conditions, recovery of the wave speed $c = \rho^{-1/2}$ is shown to be H\"older stable, even in the anisotropic setting where the speed is described by a Riemannian metric~\cite{stefanov1998stability, stefanov2005stable, bellassoued2010stability, montalto2014stable}. Furthermore, a low-pass filtered version of $c$ can be reconstructed with Lipschitz stability~\cite{liu2016lipschitz}. H\"older stability results for the damping coefficient $\sigma$, under suitable geometric assumptions, have been established in~\cite{bellassoued2017stable, stefanov2018lorentzian, bellassoued2019simultaneous}.
For more general geometries, stability estimates based on the BC method were obtained in~\cite{MR2096795}, where an abstract modulus of continuity was derived, and more recently in~\cite{bosi2022reconstruction, burago2021stability}, where a doubly logarithmic modulus of continuity was established.

Numerical implementation of the BC method for reconstructing the wave speed was initiated in~\cite{belishev1999dynamical} and further developed in~\cite{belishev2016numerical, de2018recovery, pestov2010numerical, yang2021stable}. The approaches in~\cite{belishev1999dynamical, belishev2016numerical, de2018recovery} rely on solving unstable control problems, while those in~\cite{pestov2010numerical, yang2021stable} employ stable control problems but involve target functions with exponential growth or decay, which may also lead to numerical instability.
In contrast, the linearized BC approaches proposed in~\cite{oksanen2022linearized, oksanen2024linearized} are inherently stable. In the one-dimensional setting, the BC method can be implemented in a fully stable manner~\cite{korpela2018discrete}. Recently, the BC method has also been employed for numerical conversion of wave data from the time domain to the frequency domain~\cite{yang2026formula}.

\medskip
\textbf{Contribution of the Paper.} 
The paper develops analytical tools for the linearized IBVP associated with the damped wave equation. We establish several results concerning the uniqueness, stability, and reconstruction of $\dot\sigma$ including:
\begin{itemize}
    \item \emph{A linearized Blagove\u{s}\u{c}enski\u{ı} identity with a complex-valued free parameter.}  
    Blagove\u{s}\u{c}enski\u{ı} identities play a central role in the boundary control method by relating boundary measurements to inner products of wave fields. In the context of linearized IBVPs, the authors’ earlier works~\cite{oksanen2022linearized, oksanen2024linearized} derived linearized Blagove\u{s}\u{c}enski\u{ı} identities with a real-valued free parameter for the recovery of potentials and densities, respectively. The presence of a free parameter enlarges the class of admissible test functions used to probe the unknown coefficient, thereby leading to improved stability and reconstruction properties. In the present work, we extend this approach to damped wave equations. A major technical challenge is that wave operators with damping terms are no longer self-adjoint on the standard $L^2$ spaces, in contrast to the operators considered in~\cite{oksanen2022linearized, oksanen2024linearized}. To address this issue, we derive a new version of the linearized Blagove\u{s}\u{c}enski\u{ı} identity that involves a complex-valued free parameter in Proposition~\ref{thm:linearizedBid}. This result generalizes the linearized Blagove\u{s}\u{c}enski\u{ı} identity with a real-valued free parameter obtained in our earlier work~\cite{oksanen2022linearized, oksanen2024linearized}, and forms the foundation for the subsequent uniqueness, stability, and reconstruction results.
    
    \item \emph{Uniqueness, stability estimate, and reconstruction formula in dimension $n\ge 1$ for constant background damping.}  
    While the uniqueness and stability of the nonlinear IBVP have been investigated separately in several works, this paper focuses systematically on the linearized problem. We provide unified results addressing uniqueness, stability, and reconstruction. When the background damping coefficient $\sigma_0$ is constant, we derive an explicit reconstruction formula in Proposition~\ref{thm:reconformula} for all dimensions $n \geq 1$. This formula is shown to satisfy a pointwise Lipschitz-type stability estimate in the Fourier domain in Proposition~\ref{thm:stab}. The reconstruction formula further leads to a reconstruction algorithm, summarized in Algorithm~\ref{alg:reconstruction}. The algorithm is numerically validated and assessed in dimension $n=1$.

    \item \emph{A time-domain increasing stability estimate in dimension $n \geq 3$ for non-constant background damping.}  
    By exploiting the linearized Blagove\u{s}\u{c}enski\u{ı} identity with a suitably chosen complex-valued free parameter, we derive a stability estimate for the recovery of the damping perturbation $\dot{\sigma}$. The resulting estimate combines a H\"older-type stability with a logarithmic stability component. Moreover, the logarithmic contribution can be reduced by an appropriate choice of the free parameter, yielding an estimate that is close to H\"older stability. This effect, known as the \emph{increasing stability} phenomenon, is well studied in the frequency-domain analysis of the Helmholtz equation; see, for example,~\cite{cheng2016increasing, hrycak2004increased, isakov2010increasing, isakov2016increasing, isakov2018increasing, isakov2014increasing, kow2021optimality, nagayasu2013increasing}. In this work, we establish an increasing stability estimate in the time domain for dimensions $n \geq 3$; see Theorem~\ref{thm:IncreasingStability}. An important observation is that the appropriately chosen complex-valued free parameter effectively plays the role of a frequency variable in the construction of probing test functions.
\end{itemize}

The paper is structured as follows. Section~\ref{sec:prelim} derives some preliminary results, including formulation of the damped wave equation as a first-order system and the boundedness of the linearized ND map $\dot\Lambda$. Section~\ref{sec:id} is devoted to the proof of the linearized Blagove\u{s}\u{c}enski\u{ı} identity with a complex-valued free parameter. Section~\ref{sec:control} proves a boundary control result. Section~\ref{sec:reconandstab} establishes the reconstruction formula and several stability estimates. Section~\ref{sec:numvalid} consists of numerical implementation and validation of the reconstruction formula using one-dimensional examples.

\section{Preliminary} \label{sec:prelim}

\textbf{Damped Wave Equation as a First-Order System.}
Henceforth, we will deal with the wave equation with damping as a first-order linear system. To this end, we denote $u_t:=\partial_t u$ and introduce new variables
\[p\coloneqq\sqrt{\rho}u_t,\qquad q\coloneqq\nabla u.\]
The damped wave equation can be written as the following first-order linear system:
\begin{equation}\label{eq:wavesystem}
    \partial_t\begin{pmatrix}p\\q\end{pmatrix}=\begin{pmatrix}-\sigma\rho^{-1}&\rho^{-1/2}\nabla\cdot\\\nabla[\rho^{-1/2}\cdot]&0\end{pmatrix}\begin{pmatrix}p\\q\end{pmatrix}
\end{equation}
where $\nabla\cdot$ denotes the divergence operator. Accordingly, the ND map can be written as
\begin{equation}
    \Lambda_\sigma: \nu\cdot q \mapsto \partial^{-1}_t \left(\rho^{-1/2} p \right)
\end{equation}
where $\partial^{-1}_t$ is defined as $\partial^{-1}_t f(t) = \int^t_0 f(\tau) \, \dif \tau$.
Similarly, the linearized ND map can be written as
\begin{equation}
    \dot{\Lambda}_{\dot\sigma}: \nu\cdot q_0 \mapsto \partial^{-1}_t \left(\rho^{-1/2} \dot{p} \right)
\end{equation}
where $q_0=\nabla u_0$ and $\dot{p} = \sqrt{\rho_0} \dot{u}_t$ with $u_0,\dot{u}$ the solutions of~\eqref{eq:wave0} ~\eqref{eq:wavedot}, respectively.

\medskip
\textbf{Boundedness of $\dot\Lambda$.} We begin by showing that $\dot\Lambda$ is a bounded linear operator between suitable spaces. For any real number $s\geq \frac{1}{2}$, denote by $H_{00}^s((0,2T)\times\partial\Omega)$ the closure of the set $\{f|_{(0,2T)\times\partial\Omega}: f\in C^\infty_{c}((0,2T]\times\partial\Omega)\}$ in $H^s((0,2T)\times\partial\Omega)$ with respect to the $H^s$-norm, where $H^s$ is the $L^2$-based Sobolev space of order $s$.

\begin{prop} \label{thm:Lambdadotboundedness}
    For any real number $s\geq \frac{1}{2}$, $\dot\Lambda:H_{00}^s((0,2T)\times\partial\Omega) \rightarrow H^s((0,2T)\times\partial\Omega)$ is a bounded linear operator.
\end{prop}

\begin{proof}
The linearity is clear. To show the boundedness, let us take an arbitrary $f\in C^\infty_c((0,2T]\times\partial\Omega)$ and extend it to $F\in H^{s+\frac{3}{2}}((0,2T)\times\Omega)$ such that $\partial_\nu F=f$ and $F(t,x)=0$ for any $x\in\Omega$ and $t$ close to $0$. Such $F$ can be chosen to satisfy
\[\|F\|_{H^{s+\frac{3}{2}}((0,2T)\times\Omega)}\leq C\|f\|_{H^s((0,2T)\times\partial\Omega)}.\]
Denote $v\coloneqq u_0^f-F$ where $u_0^f$ is the solution of \eqref{eq:wave0} with the Neumann boundary condition $f$, then $v$ satisfies
$$
\left\{
\begin{alignedat}{2}
\square_{\rho_0,\sigma_0}v(t,x)& = - \square_{\rho_0,\sigma_0}F(t,x)\quad&&\text{ in }(0,2T)\times\Omega,\\
    v(0,x)=\partial_tv(0,x)&=0\quad&&\text{ on } \Omega,\\
    \partial_\nu v(t,x)&=0\quad&&\text{ on }(0,2T)\times\partial\Omega.
\end{alignedat}
\right.
$$
According to the regularity estimate for the wave equation~\cite{evans1998partial},
$$
    \|v\|_{H^{s+\frac{1}{2}}((0,2T)\times\Omega)}\leq C\|\square_{\rho_0,\sigma_0}F\|_{H^{s-\frac{1}{2}}((0,2T)\times\Omega)}\leq C\|F\|_{H^{s+\frac{3}{2}}((0,2T)\times\Omega)}.
$$
Hence $u_0^f = v+F \in{H^{s+\frac{1}{2}}((0,2T)\times\Omega)}$, then $\dot{\sigma}\partial_tu_0^f\in{H^{s-\frac{1}{2}}((0,2T)\times\Omega)}$ since $\dot{\sigma}\in C^\infty(\overline{\Omega})$. The wave regularity estimate applied to \eqref{eq:wavedot} gives
$$
    \|\dot{u}^f\|_{H^{s+\frac{1}{2}}((0,2T)\times\Omega)}\leq C\|\dot{\sigma}\partial_tu_0^f\|_{H^{s-\frac{1}{2}}((0,2T)\times\Omega)}\leq C\|u_0^f\|_{H^{s+\frac{1}{2}}((0,2T)\times\Omega)}.
$$
Together with trace estimation, we get
$$
    \|\dot{\Lambda}f\|_{H^s((0,2T)\times\Omega)}\leq C\|\dot{u}^f\|_{H^{s+\frac{1}{2}}((0,2T)\times\Omega)}\leq C\|f\|_{H^s((0,2T)\times\partial\Omega)},
$$
where the constant $C$ is independent of $f$. The result follows since $C_{c}^\infty((0,2T]\times\partial\Omega)$ is dense in $H_{00}^s((0,2T)\times\partial\Omega)$ with respect to the $H^s$-norm.

\end{proof}

\section{Linearized Blagove\u{s}\u{c}enski\u{ı} Identity with a Free Complex-Valued Parameter} \label{sec:id}

The next result represents wave inner products inside $\Omega$ in terms of boundary integrals. This type of relation is known as the Blagove\u{s}\u{c}enski\u{ı} identity~\cite{blagovestchenskiiinverse}. Here, we write the identity using the variables $(p,q)$. Another proof of the same identity using integration by parts is established in~\cite{pestov2012inverse, pestov2014determining}.

\begin{prop}
Let $u^f,u^h$ be the solution of \eqref{eq:wave} with Neumann traces $f,h\in C_c^\infty((0,2T]\times\partial\Omega)$, respectively. Let $(p^f,q^f)$, $(p^h,q^h)$ denote the corresponding solutions of \eqref{eq:wavesystem}. Then
\begin{equation} \label{eq:Bidentity}
    \left\langle\begin{pmatrix}p^f(T)\\-q^f(T)\end{pmatrix},\begin{pmatrix}p^h(T)\\q^h(T)\end{pmatrix}\right\rangle_{L^2(\Omega)}=\langle f(\cdot),[\Lambda_{\sigma}h_t](2T-\cdot)\rangle_{L^2((0,T)\times\partial\Omega)}-\langle[\Lambda_{\sigma}f_t](\cdot),h(2T-\cdot)\rangle_{L^2((0,T)\times\partial\Omega)}
\end{equation}
where $\langle\cdot,\cdot\rangle$ is the $L^2$-inner product for vector-valued functions.
\end{prop}
\begin{proof}
The assumption ensures that $p^f, p^h, q^f, q^h$ are smooth functions. Define
$$
    I(t,s)\coloneqq\left\langle\begin{pmatrix}p^f(t)\\-q^f(t)\end{pmatrix},\begin{pmatrix}p^h(s)\\q^h(s)\end{pmatrix}\right\rangle_{L^2(\Omega)}.
$$
Note $I(0,s)=0$ due to the initial conditions for $u^f$ and $u^h$. We differentiate in $t$ and $s$ respectively, and utilized the system~\eqref{eq:wavesystem} to have
$$
    \partial_tI(t,s) = \left\langle\begin{pmatrix}-\sigma\rho^{-1}&\rho^{-1/2}\nabla\cdot\\-\nabla[\rho^{-1/2}\cdot]&0\end{pmatrix}\begin{pmatrix}p^f(t)\\q^f(t)\end{pmatrix},\begin{pmatrix}p^h(s)\\q^h(s)\end{pmatrix}\right\rangle_{L^2(\Omega)},
$$
and
$$    
\partial_sI(t,s) = \left\langle\begin{pmatrix}p^f(t)\\-q^f(t)\end{pmatrix},\begin{pmatrix}-\sigma\rho^{-1}&\rho^{-1/2}\nabla\cdot\\\nabla[\rho^{-1/2}\cdot]&0\end{pmatrix}\begin{pmatrix}p^h(s)\\q^h(s)\end{pmatrix}\right\rangle_{L^2(\Omega)}.
$$
Subtract these two equations to get 
$$
    \begin{aligned}
    [\partial_t-\partial_s]I(t,s) &= \langle \rho^{-1/2}p^h(s), \nabla\cdot q^f(t) \rangle_{L^2(\Omega)} + \langle q^f(t), \nabla \left( \rho^{-1/2} p^h \right)(s) \rangle_{L^2(\Omega)} \\
    & \quad - \langle q^h(s), \nabla \left( \rho^{-1/2} p^f \right)(t) \rangle_{L^2(\Omega)} - \langle \rho^{-1/2}p^f(t), \nabla\cdot q^h(s) \rangle_{L^2(\Omega)} \\
    & = \langle \rho^{-1/2}p^h(s),\nu\cdot q^f(t)\rangle_{L^2(\partial\Omega)} - \langle \rho^{-1/2}p^f(t),\nu \cdot q^h(s)\rangle_{L^2(\partial\Omega)}\\
    &= \langle \partial_s\Lambda_{\sigma} h(s), f(t)\rangle_{L^2(\partial\Omega)}-\langle \partial_t\Lambda_{\sigma} f(t), h(s)\rangle_{L^2(\partial\Omega)}
    \end{aligned}
$$
where the second equality follows from the integration-by-parts formula $\langle a, \nabla\cdot \vec{b} \rangle_{L^2(\Omega)} +  \langle \nabla a, \vec{b} \rangle_{L^2(\Omega)} = \langle a, \nu\cdot \vec{b} \rangle_{L^2(\partial\Omega)}$ for a function $a$ and a vector field $\vec{b}$.
Solve this transport equation to obtain
\begin{equation}\label{eq:Its}
    \begin{aligned}
    I(t,s) &= I(0,t+s) + \int_0^t\langle f(\tau),[\Lambda_{\sigma} h_t](t+s-\tau)\rangle_{L^2(\partial\Omega)}-\langle h(t+s-\tau),[\Lambda_{\sigma} f_t](\tau)\rangle_{L^2(\partial\Omega)}\dif \tau,
    \end{aligned}
\end{equation}
where we used the fact $\partial_t^nu^f = u^{\partial_t^n f}$ to make $\Lambda_\sigma$ and $\partial_t$ commute, which is ensured by the well-posedness of the forward problem~\eqref{eq:wave}. Insert the initial condition $I(0,t+s)=0$ and set $t=s=T$ to get
$$
    \begin{aligned}
    I(T,T) =\langle f(\cdot),[\Lambda_{\sigma}h_t](2T-\cdot)\rangle_{L^2((0,T)\times\partial\Omega)}-\langle[\Lambda_{\sigma}f_t](\cdot),h(2T-\cdot)\rangle_{L^2((0,T)\times\partial\Omega)}.
    \end{aligned}
$$
\end{proof}

For our purpose, we need a linearized version of~\eqref{eq:Bidentity}. We will derive the linearization in detail in the next proposition, where a free complex-valued parameter $\lambda\in\mathbb{C}$ is introduced. As we will see, proper choice of this parameter helps recover the high-frequency content of $\dot\sigma$. This idea of introducing a free parameter in the linearized Blagove\u{s}\u{c}enski\u{ı} identity was developed in our earlier work~\cite{oksanen2022linearized, oksanen2024linearized}.

\begin{prop} \label{thm:linearizedBid}
Let $\lambda\in\mathbb{C}$ be a complex number. If $f,h\in C_c^\infty((0,2T]\times\partial\Omega)$ satisfy
\begin{equation}\label{eq:final}
    \begin{aligned}
    &\left[\begin{pmatrix}-\sigma_0\rho_0^{-1}&\rho_0^{-1/2}\nabla\cdot\\-\nabla[\rho_0^{-1/2}\cdot]&0\end{pmatrix}+\lambda\begin{pmatrix}1&0\\0&-1\end{pmatrix}\right]\begin{pmatrix}p_0^f(T)\\q_0^f(T)\end{pmatrix}\\
    =&\left[\begin{pmatrix}-\sigma_0\rho_0^{-1}&\rho_0^{-1/2}\nabla\cdot\\-\nabla[\rho_0^{-1/2}\cdot]&0\end{pmatrix}+\lambda\begin{pmatrix}1&0\\0&-1\end{pmatrix}\right]\begin{pmatrix}p_0^h(T)\\q_0^h(T)\end{pmatrix}=\begin{pmatrix}0\\0\end{pmatrix},
    \end{aligned}
\end{equation}
then the following equality holds:
\begin{equation}\label{eq:control}
    \begin{aligned}
    \langle p_0^f(T),p_0^h(T)\rangle_{L^2(\Omega,\dot\sigma\rho_0^{-1}\dif x)}  =& -\langle f(T),[\dot{\Lambda} h_t](T)\rangle_{L^2(\partial\Omega)}\\
    &-\langle f(t),[\dot{\Lambda}h_{tt}](2T-t)\rangle_{L^2((0,T)\times\partial\Omega)}+\langle[\dot{\Lambda}f_t](t),h_t(2T-t)\rangle_{L^2((0,T)\times\partial\Omega)}\\
    &-\lambda\langle f(t),[\dot{\Lambda}h_t](2T-t)\rangle_{L^2((0,T)\times\partial\Omega)}+\lambda\langle[\dot{\Lambda}f_t](t),h(2T-t)\rangle_{L^2((0,T)\times\partial\Omega)}
    \end{aligned}
\end{equation}
where $\langle \cdot, \cdot \rangle_{L^2(\Omega,\dot\sigma\rho_0^{-1}\dif x)}$ denotes the weighted $L^2$-inner product with weight $\dot{\sigma}\rho^{-1}_0$.
We write $\rho=\rho_0$ to indicate that the known density is not linearized.
\end{prop}

\begin{proof}
Consider the asymptotic expansion of $I(t,s)$ and $\partial_t I(t,s)$ in small $\varepsilon>0$. The first-order terms in $\varepsilon$ are
$$
    \dot{I}(t,s)\coloneqq\left\langle\begin{pmatrix}\dot{p}^f(t)\\-\dot{q}^f(t)\end{pmatrix},\begin{pmatrix}p_0^h(s)\\q_0^h(s)\end{pmatrix}\right\rangle_{L^2(\Omega)}+\left\langle\begin{pmatrix}p_0^f(t)\\-q_0^f(t)\end{pmatrix},\begin{pmatrix}\dot{p}^h(s)\\\dot{q}^h(s)\end{pmatrix}\right\rangle_{L^2(\Omega)}
$$
$$
    \begin{aligned}
    \partial_t\dot{I}(t,s)\coloneqq&\left\langle\begin{pmatrix}-\dot{\sigma}\rho_0^{-1}&0\\0&0\end{pmatrix}\begin{pmatrix}p_0^f(t)\\q_0^f(t)\end{pmatrix},\begin{pmatrix}p_0^h(s)\\q_0^h(s)\end{pmatrix}\right\rangle_{L^2(\Omega)}+\left\langle\begin{pmatrix}-\sigma_0\rho_0^{-1}&\rho_0^{-1/2}\nabla\cdot\\-\nabla[\rho_0^{-1/2}\cdot]&0\end{pmatrix}\begin{pmatrix}p_0^f(t)\\q_0^f(t)\end{pmatrix},\begin{pmatrix}\dot{p}^h(s)\\\dot{q}^h(s)\end{pmatrix}\right\rangle_{L^2(\Omega)}\\&+\left\langle\begin{pmatrix}-\sigma_0\rho_0^{-1}&\rho_0^{-1/2}\nabla\cdot\\-\nabla[\rho_0^{-1/2}\cdot]&0\end{pmatrix}\begin{pmatrix}\dot{p}^f(t)\\\dot{q}^f(t)\end{pmatrix},\begin{pmatrix}p_0^h(s)\\q_0^h(s)\end{pmatrix}\right\rangle_{L^2(\Omega)},
    \end{aligned}
$$
Apply the integration by parts to the last term to get
$$
    \begin{aligned}
        &\left\langle\begin{pmatrix}-\sigma_0\rho_0^{-1}&\rho_0^{-1/2}\nabla\cdot\\-\nabla[\rho_0^{-1/2}\cdot]&0\end{pmatrix}\begin{pmatrix}\dot{p}^f(t)\\\dot{q}^f(t)\end{pmatrix},\begin{pmatrix}p_0^h(s)\\q_0^h(s)\end{pmatrix}\right\rangle_{L^2(\Omega)} \\
        =&-\langle\dot{p}^f(t),p_0^h(s)\rangle_{L^2(\Omega,\sigma_0\rho_0^{-1}\dif x)}+\langle\rho_0^{-1/2}\nabla\cdot\dot{q}^f(t),p_0^h(s)\rangle_{L^2(\Omega)} - \langle\nabla[\rho_0^{-1/2}\dot{p}^f(t)],q_0^h(s)\rangle_{L^2(\Omega)}\\
        =&-\langle\dot{p}^f(t),p_0^h(s)\rangle_{L^2(\Omega,\sigma_0\rho_0^{-1}\dif x)}-\langle\dot{q}^f(t),\nabla[\rho_0^{-1/2}p_0^h(s)]\rangle_{L^2(\Omega)} + \langle\dot{p}^f(t),\rho_0^{-1/2}\nabla\cdot q_0^h(s)\rangle_{L^2(\Omega)}\\
        &+\langle \nu\cdot\dot{q}^f(t),\rho_0^{-1/2}p^h_0(s)\rangle_{L^2(\partial\Omega)}-\langle\rho_0^{-1/2}\dot{p}^f(t),\nu\cdot q_0^h(s)\rangle_{L^2(\partial\Omega)}\\
        =&\left\langle\begin{pmatrix}\dot{p}^f(t)\\\dot{q}^f(t)\end{pmatrix},\begin{pmatrix}-\sigma_0\rho_0^{-1}&\rho_0^{-1/2}\nabla\cdot\\-\nabla[\rho_0^{-1/2}\cdot]&0\end{pmatrix}\begin{pmatrix}p_0^h(s)\\q_0^h(s)\end{pmatrix}\right\rangle_{L^2(\Omega)}\\
        &+\langle\nu\cdot\dot{q}^f(t),\rho_0^{-1/2}p^h_0(s)\rangle_{L^2(\partial\Omega)}-\langle\rho_0^{-1/2}\dot{p}^f(t),\nu\cdot q_0^h(s)\rangle_{L^2(\partial\Omega)}.
    \end{aligned}    
$$
Since the states $(p_0^f,q_0^f)$, $(p_0^h,q_0^h)$ at $t=T$ satisfy \eqref{eq:final}, we have
\begin{align*}
    \partial_t\dot{I}(T,T) + \lambda\dot{I}(T,T) & = -\langle p_0^f(T),p_0^h(T)\rangle_{L^2(\Omega,\dot\sigma\rho_0^{-1}\dif x)} +\langle\nu\cdot\dot{q}^f(T),\rho_0^{-1/2}p^h_0(T)\rangle_{L^2(\partial\Omega)}-\langle\rho_0^{-1/2}\dot{p}^f(T), \nu\cdot q_0^h(T)\rangle_{L^2(\partial\Omega)} \\
    & = -\langle p_0^f(T),p_0^h(T)\rangle_{L^2(\Omega,\dot\sigma\rho_0^{-1}\dif x)} -\langle [\dot{\Lambda} f_t](T), h(T)\rangle_{L^2(\partial\Omega)}
\end{align*}
where the second equality follows from the facts that $\nu\cdot \dot{q}^f(T) = \partial_\nu \dot{u}^f(T)=0$ (see~\eqref{eq:wavedot}) and 
$\rho_0^{-1/2}\dot{p}^f = \partial_t \dot{\Lambda} f = \dot{\Lambda} f_t$.

On the other hand, we can use~\eqref{eq:Its} to find the first-order term of the same asymptotic expansion:
$$
\begin{aligned}
    \partial_t\dot{I}(T,T) + \lambda\dot{I}(T,T) =& \langle f(T),[\dot{\Lambda} h_t](T)\rangle_{L^2(\partial\Omega)}-\langle h(T),[\dot{\Lambda} f_t](T)\rangle_{L^2(\partial\Omega)}\\
    &+\langle f(t),[\dot{\Lambda}h_{tt}](2T-t)\rangle_{L^2((0,T)\times\partial\Omega)}-\langle[\dot{\Lambda}f_t](t),h_t(2T-t)\rangle_{L^2((0,T)\times\partial\Omega)}\\
    &+\lambda\langle f(t),[\dot{\Lambda}h_t](2T-t)\rangle_{L^2((0,T)\times\partial\Omega)}-\lambda\langle[\dot{\Lambda}f_t](t),h(2T-t)\rangle_{L^2((0,T)\times\partial\Omega)}.
    \end{aligned}
$$
Comparing these two representations of $\partial_t\dot{I}(T,T) + \lambda\dot{I}(T,T)$ gives~\eqref{eq:control}.
\end{proof}
\begin{remark}
As $\lambda\in\mathbb{C}$ is a complex number, we allow the wave solution $u$ and the Neumann boundary condition $f$ to be complex-valued functions. Accordingly, the inner product $\langle\cdot,\cdot\rangle$ is extended to complex-valued functions through complexification.
\end{remark}

\section{A Boundary Control Result} \label{sec:control}

In this section, we show that a smooth Neumann boundary control exists for smooth target data. The proof adapts the generic treatment in~\cite[Theorem 5.1]{ervedoza2010systematic}. The result generalizes~\cite[Proposition 4]{oksanen2022linearized} by controlling both $u^f_0(T)$ and $\partial_t u^f_0(T)$ with a stability estimate. Here, we view $(\overline{\Omega}, \rho_0 dx^2)$ as a smooth Riemannian manifold with boundary.

\begin{prop}\label{prop:control}
Let $\rho_0\in C^\infty(\overline{\Omega})$ be strictly positive and $\sigma_0\in C^\infty(\overline{\Omega})$ be non-negative. Suppose that all maximal geodesics on $(\overline{\Omega},g)$ have length strictly less than $T>0$. Then for any $\phi,\psi\in C^\infty(\overline{\Omega})$ there exists a Neumann data $f\in C_c^\infty((0,T]\times\partial\Omega)$ such that
$$
    u_0^f(T) = \phi,\qquad \partial_tu_0^f(T) = \psi\qquad\textup{ in }\Omega,
$$
where $u_0$ is the solution of \eqref{eq:wave0}. Moreover, there is a constant $C>0$, independent of $\phi$ and $\psi$, such that
\begin{equation}\label{eq:controlestimate}
    \|f\|_{H^2((0,T)\times\partial\Omega)}\leq C(\|\phi\|_{H^4(\Omega)}+\|\psi\|_{H^3(\Omega)})
\end{equation}
\end{prop}

\begin{proof}
Similar to the proof of \cite[Proposition 4]{oksanen2022linearized}, we have the following results
\begin{enumerate}
    \item There exists a small constant $\delta>0$ such that the maximal geosedics have length less than $T^*\coloneqq T-2\delta$.
    \item We extend $\rho_0,\sigma_0$ smoothly to $\mathbb{R}^n$, there exists a compact domain $\mathcal{K}$ with smooth boundary such that $\overline{\Omega}$ is contained in the interior of $\mathcal{K}$ and with Riemannian metric extended by tensor $g=\rho_0\dif x^2$.
    \item There exists an open set $\omega_0$ such that $\overline{\omega_0}\subset\mathcal{K}\setminus\overline{\Omega}$ and that all geodesics $\gamma_{x,\xi}$ with $(x,\xi)\in S\overline{\Omega}$ intersect $\omega_0$ in time $T^*$, where $S\overline{\Omega}$ denote the unit sphere bundle over the closure of $\Omega$.
\end{enumerate}
We choose $\eta\in C_c^\infty((0,T))$ such that $0\leq\eta\leq1$, $\eta=1$ on $[\delta,T-\delta]$ and $\chi\in C_c^\infty(\mathcal{K}\setminus\overline{\Omega})$ such that $\chi=1$ on $\omega_0$. We choose extensions $\phi,\psi\in C_c^\infty(\mathcal{K})$ such that
\[\|\phi\|_{H^4(\mathcal{K})}\leq C\|\phi\|_{H^4(\Omega)},\quad\|\psi\|_{H^3(\mathcal{K})}\leq C\|\psi\|_{H^3(\Omega)},\]
where $C$ is independent of $\phi,\psi$. Following \cite[Theorem 5.1]{ervedoza2010systematic}, there exists $Y\in C^\infty((0,T)\times\mathcal{K})$ such that
\begin{equation}\label{eq:controlwave}
    \begin{cases}
        \square_{\rho_0,\sigma_0} v=\eta\chi Y,\qquad\textup{ in }(0,T)\times\mathcal{K}\\
        v|_{x\in\partial\mathcal{K}}=0,\\
        v|_{t=T}=\phi,\partial_tv|_{t=T}=\psi,\\
        v|_{t=0}=0,\partial_tv|_{t=0}=0,
    \end{cases}
\end{equation}
where $Y$ satisfies
$$
    \begin{cases}
        \square_{\rho_0,-\sigma_0} Y=0,\qquad\textup{ in }(0,T)\times\mathcal{K}\\
        Y|_{x\in\partial\mathcal{K}}=0,\\
        Y|_{t=0}=Y_0,\partial_tY|_{t=0}=Y_1,\\
    \end{cases}
$$
for a suitable initial condition $(Y_0,Y_1)$. Following the proof of \cite[Proposition 4]{oksanen2022linearized}, we also obtain the estimate
\[\|Y\|_{H^{3}((0,T)\times\mathcal{K})}\leq C\|(Y_0,Y_1)\|_{H^3(\mathcal{K})\times H^2(\mathcal{K})}\leq C\|(\phi,\psi)\|_{H^{4}(\mathcal{K})\times H^{3}(\mathcal{K})},\]
where $C$ is independent of $\phi,\psi$.

We set $f=\partial_\nu v|_{x\in\partial\Omega}$. Since $\eta$ is compact supported, $v$ satisfies vanishing initial condition and $f(t,x)=0$ when $t$ is sufficiently small, thus $v|_{(0,T)\times\Omega}$ solves~$\eqref{eq:wave0}$ with the Neumann boundary condition $f$. According to the Trace Theorem, we have
\[\|f\|_{H^2((0,T)\times\partial\Omega)}\leq C\|v\|_{H^{3+\frac{1}{2}}((0,T)\times\Omega)}\leq C\|v\|_{H^{4}((0,T)\times\mathcal{K})}.\]

Moreover, considering $v$ as solution of \eqref{eq:controlwave} without the constraint at $t=T$, the regularity estimate gives
\[\|v\|_{H^{4}((0,T)\times\mathcal{K})}\leq C\|\eta\chi Y\|_{H^{3}((0,T)\times\mathcal{K})}\leq C\|Y\|_{H^{3}((0,T)\times\mathcal{K})}.\]
We get \eqref{eq:controlestimate} by combining all the inequalities.
\end{proof}

\begin{remark} \label{thm:antideri}
Using the first-order system representation of the damped wave equation, the boundary control result implies the following: For $\Psi\in C^\infty(\overline{\Omega})$ and $\Phi\in C^\infty(\overline{\Omega};\mathbb{R}^n)$ with $\nabla\times\Phi=0$ (this condition ensures $\Phi$ is in the range of the gradient operator), the boundary control equations
$$
    p^f(T) = \Psi, \qquad q^f(T) =\Phi \qquad \text{ in } \Omega
$$
admits solutions. Indeed, let $\Pi\in C^\infty(\overline{\Omega})$ be such that $\nabla\Pi=\Phi$, the equations for $p,q$ translate to
$$
    u^f(T) = \nabla^{-1} \Phi := \Pi + C_q, \qquad
    \partial_t u^f(T) = \rho^{-1/2} p^f(T) = \rho^{-1/2} \Psi  \qquad \text{ in } \Omega
$$
which admit at least a solution $f \in C^\infty_c((0,T]\times\partial\Omega)$ by Proposition~\ref{prop:control}. 
Moreover, there is a constant $C>0$ such that
$$
        \|f\|_{H^2((0,T)\times\partial\Omega)}\leq C(\|\nabla^{-1}\Phi\|_{H^4(\Omega)}+\|\Psi\|_{H^3(\Omega)})
$$
as $\rho\in C^\infty(\overline{\Omega})$ is bounded from above and below.
Here, we abuse the notation and use $\nabla^{-1} \Phi$ to denote any smooth function whose gradient is $\Phi$. Any such function can be written as $\Pi + C_q$ for some constant $C_q$. 
Note that the boundary control solution $f=f_{C_q}$ also depends on the choice of the constant $C_q$.

\end{remark}

\section{Reconstruction and Stability} \label{sec:reconandstab}

In this section, we will derive several results concerning the reconstruction and stability of $\dot\sigma$. We will work only with $p^f_0(T), p^h_0(T)$ by eliminating $q^f_0(T), q^h_0(T)$ in conditions~\eqref{eq:final}. Specifically, let 
$p_0$ denote $p^f_0$ or $p^h_0$, and $q_0$ denote $q^f_0$ or $q^h_0$. We expand the first row in~\eqref{eq:final} to get
$$
\nabla \cdot q_0(T) = \sqrt{\rho_0} \left( \frac{\sigma_0}{\rho_0} - \lambda \right) p_0(T).
$$
We expand the second row in~\eqref{eq:final} to get
$$
\lambda q_0(T) = -\nabla\left( \frac{1}{\sqrt{\rho_0}}p_0(T) \right).
$$
Combining these equations, we see that
$$
\lambda\nabla\cdot q_0(T) = -\Delta \left( \frac{1}{\sqrt{\rho_0}} p_0(T) \right) = \lambda \sqrt{\rho_0} \left( \frac{\sigma_0}{\rho_0} - \lambda \right) p_0(T).
$$
Hence
\begin{equation} \label{eq:peq}
\frac{1}{\sqrt{\rho_0}} \Delta \left( \frac{1}{\sqrt{\rho_0}} p_0(T) \right) + \lambda \left( \frac{\sigma_0}{\rho_0} - \lambda \right) p_0(T) = 0, \qquad p_0 = p^f_0 \text{ or } p^h_0
\end{equation}

From now on, we will take the background density $\rho_0\equiv 1$ and focus on the recovery of $\dot\sigma$. The discussion is separated into two cases: (1) $\sigma_0$ is a constant. In this case, we derive an explicit reconstruction formula for $\dot{\sigma}$, from which the stability estimate and reconstruction algorithm follow. (2) $\sigma_0=\sigma_0(x)$ is spatially varying but small in the Sobolev scale. In this case, we obtain an increasing stability estimate for the determination of $\dot{\sigma}$.

\subsection{Constant $\sigma_0$} \label{sec:ConstBGDamp}

When $\rho_0\equiv 1$ and $\sigma_0$ is constant, the equation~\eqref{eq:peq} reduces to
$$
[\Delta + \lambda(\sigma_0-\lambda)] p^f_0(T) = [\Delta + \lambda(\sigma_0-\lambda)] p^h_0(T) = 0.
$$
In this case, the corresponding $q_0^f(T), q_0^h(T)$ are
\begin{equation} \label{eq:qcond}
    q^f_0(T) = -\frac{1}{\lambda} \nabla p^f_0(T), \qquad
    q^h_0(T) = -\frac{1}{\lambda} \nabla p^h_0(T).
\end{equation}
The equations for $p^f_0(T), p^h_0(T)$ can be turned into Helmholtz equations with proper choice of $\lambda\in\mathbb{C}$. Indeed, let $k\geq 0$ be an arbitrary non-negative real number, and choose $\lambda$ so that $\lambda(\sigma_0-\lambda \rho_0)=k^2$, that is,
\[\lambda=
\begin{cases}
\frac{\sigma_0 + i \sqrt{4 k^2-\sigma_0^2} }{2\rho_0},

& \text{ if }k>\frac{\sigma_0}{2 },\\
\frac{\sigma_0 + \sqrt{\sigma_0^2 - 4 k^2} }{2},

& \text{ if }0\leq k\leq\frac{\sigma_0}{2}.
\end{cases}
\] 
The resulting Helmholtz equations are:
\[[\Delta+k^2]p^f_0(T)=[\Delta+k^2]p^h_0(T)=0.\]

Recall that for any $\theta\in\mathbb{S}^{n-1}$ and any $k\geq 0$, the function $\phi_\theta(x):=\exp(ik\theta\cdot x)$ is a Helmholtz solution. 
According to Proposition~\ref{prop:control}, there exist Neumann boundary conditions $f,h\in C_c^\infty((0,T]\times\partial\Omega)$ such that
\begin{equation}\label{eq:helmholtzsol}
    p_0^f(T) = p_0^h(T) = \exp(ik\theta\cdot x).
\end{equation}
In this case,
$$
q_0^f(T) = q_0^h(T) = 
-\frac{1}{\lambda} \nabla p^f_0(T) =-\frac{ik}{\lambda}\exp(ik\theta\cdot x)\cdot\theta.
$$

\begin{remark}
When $\lambda=k=\sigma_0=0$, we choose $q_0(T)$ to be $-\theta$.
\end{remark}

With this choice of $\lambda$, we can calculate the Fourier transform $\hat{\dot\sigma}$ as follows.

\begin{prop} \label{thm:reconformula}
Suppose $\rho_0\equiv 1$ and $\sigma_0$ is constant. The Fourier transform of $\dot{\sigma}$ can be constructed as follows:
\begin{equation}\label{eq:reconstruction1d}
    \begin{aligned}
    \hat{\dot\sigma}(2k\theta)  =& -\langle f(T),[\dot{\Lambda} h_t](T)\rangle_{L^2(\partial\Omega)}\\
    &-\langle f(t),[\dot{\Lambda}h_{tt}](2T-t)\rangle_{L^2((0,T)\times\partial\Omega)}+\langle[\dot{\Lambda}f_t](t),h_t(2T-t)\rangle_{L^2((0,T)\times\partial\Omega)}\\
    &-\lambda\langle f(t),[\dot{\Lambda}h_t](2T-t)\rangle_{L^2((0,T)\times\partial\Omega)}+\lambda\langle[\dot{\Lambda}f_t](t),h(2T-t)\rangle_{L^2((0,T)\times\partial\Omega)},
    \end{aligned}
\end{equation}
where $f,h\in C_c^\infty((0,T]\times\partial\Omega)$ are solutions of the equation \eqref{eq:helmholtzsol}.
\end{prop}
\begin{proof}
Inserting~\eqref{eq:helmholtzsol} into~\eqref{eq:control} yields the identity, which recovers the Fourier transform of $\dot{\sigma}$ everywhere since $k\geq0$ and $\theta\in\mathbb{S}^{n-1}$ are arbitrary.
\end{proof}

\begin{prop} \label{thm:stab}
    The reconstruction~\eqref{eq:reconstruction1d} satisfies the following pointwise stability estimate in the Fourier domain:
    $$
    | \hat{\dot\sigma}(2k\theta) | \leq C_\lambda\|\dot{\Lambda}\|_{H^2_{00}((0,T)\times\partial\Omega)\to H^2((0,T)\times\partial\Omega)} 
    $$
    for some constant $C_\lambda>0$ that depends on $\lambda$ but not $\dot{\sigma}$.
\end{prop}

\begin{proof}
We introduce the time-reversal operator $\mathcal{R}$ defined as
$$
\mathcal{R} f(t,x) := f(2T-t,x), \qquad t\in (0,T), \quad x\in \partial\Omega.
$$
It is clear that $\mathcal{R}$ is bounded on $H^s((0,T)\times\partial\Omega)$ for any $s\geq 0$.
We estimate using~\eqref{eq:reconstruction1d} and the Cauchy-Schwarz inequality to get
$$
    \begin{aligned}
    |\hat{\dot\sigma}(2k\theta)|
    \leq&\|f\|_{L^2((0,T)\times\partial\Omega)}\|\mathcal{R}\dot{\Lambda}h_{tt}\|_{L^2((0,T)\times\partial\Omega)}+\|\dot{\Lambda}f_t\|_{L^2((0,T)\times\partial\Omega)}\|\mathcal{R} h_t\|_{L^2((0,T)\times\partial\Omega)}\\
    &+|\lambda|\|f\|_{L^2((0,T)\times\partial\Omega)}\|\mathcal{R}\dot{\Lambda}h_t\|_{L^2((0,T)\times\partial\Omega)}+|\lambda|\|\dot{\Lambda}f_t\|_{L^2((0,T)\times\partial\Omega)}\|\mathcal{R}h\|_{L^2((0,T)\times\partial\Omega)}\\
    &+\|f(T)\|_{L^2(\partial\Omega)}\|\dot{\Lambda} h_t](T)\|_{L^2(\partial\Omega)}\\
    \leq&(1+|\lambda|)\|f\|_{L^2((0,T)\times\partial\Omega)}\|\mathcal{R}\dot{\Lambda}h\|_{H^2((0,T)\times\partial\Omega)}+(1+|\lambda|)\|\dot{\Lambda}f_t\|_{L^2((0,T)\times\partial\Omega)}\|\mathcal{R}h\|_{H^1((0,T)\times\partial\Omega)}\\
    &+\|f\|_{H^1((0,T)\times\partial\Omega)}\|\dot{\Lambda} h\|_{H^2((0,T)\times\partial\Omega)}\\
    \leq&(2+|\lambda|)\|f\|_{H^1((0,T)\times\partial\Omega)}\|\dot{\Lambda}h\|_{H^2((0,T)\times\partial\Omega)}+(1+|\lambda|)\|\dot{\Lambda}f\|_{H^2((0,T)\times\partial\Omega)}\|h\|_{H^1((0,T)\times\partial\Omega)}.
    \end{aligned}
$$
As $f,h\in C_c^\infty((0,T]\times\partial\Omega)$ in~\eqref{eq:reconstruction1d}, Proposition~\ref{thm:Lambdadotboundedness} implies
$$
\begin{aligned}
    |\hat{\dot\sigma}(2k\theta)|
    \leq&(2+|\lambda|)\|f\|_{H^1((0,T)\times\partial\Omega)}\|\dot{\Lambda}\|_{H_{00}^2((0,T)\times\partial\Omega)\to H^2((0,T)\times\partial\Omega)}\|h\|_{H^2((0,T)\times\partial\Omega)}\\
    &+(1+|\lambda|)\|\dot{\Lambda}\|_{H_{00}^2((0,T)\times\partial\Omega)\to H^2((0,T)\times\partial\Omega)}\|f\|_{H^2((0,T)\times\partial\Omega)}\|h\|_{H^1((0,T)\times\partial\Omega)}\\
    \leq& C_\lambda\|\dot{\Lambda}\|_{H_{00}^2((0,T)\times\partial\Omega)\to H^2((0,T)\times\partial\Omega)}
    \end{aligned}
$$
where
\begin{equation} \label{eq:Clambda}
    \begin{aligned}
    C_\lambda&\coloneqq (3+2|\lambda|)\|f\|_{H^2((0,2T)\times\partial\Omega)}\|h\|_{H^2((0,2T)\times\partial\Omega)}\\
    &\leq C(3+2|\lambda|)(\|\exp(ik\theta\cdot x)\|_{H^3(\Omega)}+\|\nabla^{-1}\frac{ik}{\lambda}\exp(ik\theta\cdot x)\cdot\theta\|_{H^4(\Omega)})^2\\
    &\leq C(3+2|\lambda|)(|k|^3+\frac{|k|^4}{|\lambda|}+C_q)^2,
    \end{aligned}    
\end{equation}

and $C_q$ is a constant generated from the anti-derivative $\nabla^{-1}$, see Remark~\ref{thm:antideri}.
\end{proof}

\begin{remark}
For complex-valued $f,h$, the inner product $\langle f,g\rangle$ can be viewed as the Hermitian product between $f$ and $\Bar{g}$, so the Cauchy-Schwarz inequality
\[|\langle f,g\rangle|\leq\|f\|_2\|g\|_2\]
remains valid, where the 2-norms are induced by the Hermitian product.
\end{remark}

\subsection{Non-constant $\sigma_0$}

When $\sigma_0=\sigma_0(x)$ is spatially-varying, the equation~\eqref{eq:peq} is no longer the Helmholtz equation. Nonetheless, if we choose $\lambda = -ik$ with $k\geq 0$ an arbitrary non-negative constant, the equation~\eqref{eq:peq} becomes the following Schr\"odinger equation:
\begin{equation}\label{eq:perturbedHelmholtz}
[-\Delta - k^2 + ik\sigma_0] p_0^f(T) = [-\Delta - k^2  + ik\sigma_0] p_0^h(T) = 0.
\end{equation}

In this case, we will employ another class of solutions known as the Complex Geometric Optics solutions (\textit{CGO solutions})~\cite{sylvester1987global}. They are functions of the form
$$
\phi(x) = e^{i\zeta\cdot x}(1+r(x)).
$$
where $\zeta\in\mathbb{C}^n$ satisfies $\zeta\cdot\zeta=k^2$ and $r(x)$ satisfies
$$
    -\Delta r - 2i\zeta\cdot\nabla r + ik\sigma_0(1+r) = 0.
$$
This equation for $r$ ensures that $\phi$ satisfies the Schr\"odinger equation~\eqref{eq:perturbedHelmholtz}. Moreover, the next result, proved in~\cite[Proposition 3.2]{isakov2016increasing}, shows that $r$ decays in the Sobolev spaces as $|\zeta|\rightarrow\infty$.

\begin{prop}[{\cite[Proposition 3.2]{isakov2016increasing}}]\label{prop:CGOregularity}
Let $\zeta\in\mathbb{C}^n$ $(n\geq 3)$ satisfy $\zeta\cdot\zeta=k^2$. Let $s>\frac{n}{2}$. 
There exist positive constants $C_0$, $C_1$, depending on $s$ and $\Omega$, such that if $C_0k\|\sigma_0\|_{H^s(\Omega)}\leq |\zeta|$, then 
$$
    \phi(x) = e^{i\zeta\cdot x}(1+r(x))
$$
satisfies the Sch\"odinger equation~\eqref{eq:perturbedHelmholtz} with the estimate
\begin{equation} \label{eq:restimate}
    \|r\|_{H^s(\Omega)}\leq \frac{C_1}{|\zeta|}\|ik\sigma_0\|_{H^s(\Omega)}
\end{equation}
\end{prop}

From now on, we assume $n\geq 3$ in order to construct CGO solutions.
Let $\xi\in\mathbb{R}^{n}$ be an arbitrary vector. Choose another two unit vectors $e(1),e(2)\in\mathbb{S}^{n-1}$ such that $\{\xi,e(1),e(2)\}$ forms an orthogonal set. 
Let $R>0$ be an arbitrary positive number with $k^2+\frac{R^2}{2}\geq\frac{|\xi|^2}{4}$. We will eventually take $R$ to be sufficiently large. Introduce two complex vectors
\begin{align*}
    \zeta(1)&:=-\frac{1}{2}\xi + i\frac{R}{\sqrt{2}}e(1) + \sqrt{k^2+\frac{R^2}{2}-\frac{|\xi|^2}{4}}e(2),\\
    \zeta(2)&:=-\frac{1}{2}\xi - i\frac{R}{\sqrt{2}}e(1) - \sqrt{k^2+\frac{R^2}{2}-\frac{|\xi|^2}{4}}e(2).
\end{align*}
Note that the construction of $\zeta_1, \zeta_2$ fulfills the following conditions: 
\[\zeta(1)+\zeta(2)=-\xi,\quad\zeta(j)\cdot\zeta(j)=k^2,\quad|\zeta(j)|^2=R^2+k^2,\quad\textup{for }j=1,2.\]

Note that if 
$$
\|\sigma_0\|_{H^s(\Omega)}\leq\frac{1}{C_0}
$$
then we automatically have

\[|\zeta(j)|\geq k\geq C_0k\|\sigma_0\|_{H^s(\Omega)},\]
Proposition~\ref{prop:CGOregularity} asserts the existence of CGO solutions
\[\phi_j(x)=e^{i\zeta(j)\cdot x}(1+r_j(x)),\quad \text{ with } \|r_j\|_{H^s(\Omega)}\leq \frac{C_1}{|\zeta(j)|}\|ik\sigma_0\|_{H^s(\Omega)}\leq\frac{C_1}{C_0}.\]

The following lemma gives an upper bound for the $H^s$-norm of the CGO solution $\phi_j$.

\begin{lemma} \label{thm:CGOHs}
    Suppose $s>\frac{n}{2}$ $(n\geq 3)$ and $\|\sigma_0\|_{H^s(\Omega)}\leq\frac{1}{C_0}$. With the choice of $\zeta(j)$ ($j=1,2$) as above, we have
        $$
    \|\phi_j\|_{H^s(\Omega)} \leq C \left(|\Omega|^{\frac{1}{2}}+\frac{C_1}{C_0}\right) (R^2+k^2)^{\frac{s}{2}} e^{\frac{R}{\sqrt 2}}
    $$
    where $C_0, C_1$ are the constants in Proposition~\ref{prop:CGOregularity} and $|\Omega|$ denotes the Lebesgue measure of $\Omega$ in $\mathbb{R}^n$.
\end{lemma}
\begin{proof}
As $s>\frac{n}{2}$, the space $H^s(\Omega)$ is a Banach algebra, hence
$$
\|\phi_j\|_{H^s(\Omega)} \leq \|e^{i\zeta(j)\cdot x}\|_{H^s(\Omega)} \|1+r_j\|_{H^s(\Omega)} \leq \left(|\Omega|^{\frac{1}{2}}+\frac{C_1}{C_0}\right) \|e^{i\zeta(j)\cdot x}\|_{H^s(\Omega)}.
$$
If $s$ is an integer, we write $D_l:=\frac{1}{i}\partial_{x_l}$ and estimate 
\begin{align*}
    \|e^{i\zeta(j)\cdot x}\|_{H^s(\Omega)}^2=
    &\sum_{l=0}^s\sum_{|\alpha|=l}\|D^\alpha e^{i\zeta(j)\cdot x}\|_{L^2(\Omega)}^2=\sum_{l=0}^s\sum_{\sum_{m=1}^n\alpha_m=l}\left(\prod_{m=1}^n|\zeta(j)_m|^{2\alpha_m}\right)\| e^{-\textup{Im}\zeta(j)\cdot x}\|_{L^2(\Omega)}^2\\
\leq & C\sum_{l=0}^s |\zeta(j)|^{2l} e^{\sqrt{2}R}=C\sum_{l=0}^s (R^2+k^2)^l e^{\sqrt{2}R}\leq C (R^2+k^2)^s e^{\sqrt{2}R}
\end{align*}
where $C$ only depends on $s,n,\Omega$. This proves the lemma for integer $s$. The estimate for non-integer $s$ follows from interpolation.
\end{proof}

As an intermediate step, we derive a pointwise estimate for the Fourier transform $\hat{\dot\sigma}$.
For simplicity we denote $\delta\coloneqq\|\dot{\Lambda}\|_{H_{00}^2((0,T)\times\partial\Omega)\to H^2((0,T)\times\partial\Omega)}$.

\begin{lemma} \label{thm:pointwiseest}
Suppose $\|\sigma_0\|_{H^s(\Omega)}\leq\frac{1}{C_0}$ with $s>\frac{n}{2}$ $(n\geq 3)$. There exists a constant $C>0$, independent of $k$ and $\delta$, such that
$$
    |\hat{\dot{\sigma}}(\xi)|\leq
    \begin{cases}
        C\frac{(3+2k)(1+k)^2}{k^2}\left((R_0^2+k^2)^{\max(s,4)}+C_q^2\right)e^{\sqrt{2}R_0}\delta + \frac{C_2k}{\sqrt{R_0^2+k^2}}\|\dot{\sigma}\|_{H^{-s}(\Omega)}&|\xi|\leq R_0+k,\\
        C\frac{(3+2k)(1+k)^2}{k^2}\left((|\xi|^2+k^2)^{\max(s,4)}+C_q^2\right)e^{\sqrt{2}|\xi|}\delta + \frac{C_2k}{\sqrt{|\xi|^2+k^2}}\|\dot{\sigma}\|_{H^{-s}(\Omega)}&|\xi|> R_0+k,\\
    \end{cases}
$$
where $R_0$ is an arbitrary fixed positive number, and the constant $C_2:=C_1 (2+\frac{C_1}{C_0})\|\sigma_0\|_{H^s(\Omega)}$ ($C_0,C_1$ are the constants from Proposition~\ref{prop:CGOregularity}).
\end{lemma}
\begin{proof}
From Proposition~\ref{prop:control}, there exists Neumann boundary controls $f_j$ such that $p^{f_j}_0(T) = \partial_t u_0^{f_j}(T)=\phi_j$, $j=1,2$. Inserting these into~\eqref{eq:control} with $\rho_0\equiv 1$ gives
$$
\int_\Omega \dot\sigma \phi_1 \phi_2 \,\dif x = \text{ RHS }.
$$
An upper bound for the right hand side (RHS) can be obtained in the same manner as in the proof of Proposition~\ref{thm:stab}. Therefore,
$$
    \left|\int_{\Omega}\dot\sigma\phi_1\phi_2\dif x\right| = |\text{RHS}| \leq C_\lambda \|\dot{\Lambda}\|_{H_{00}^2((0,T)\times\partial\Omega)\to H^2((0,T)\times\partial\Omega)} = C_\lambda \delta
$$
where the constant $C_\lambda$ given by~\eqref{eq:Clambda} satisfies
 $$
    \begin{aligned}
    C_\lambda = & (3+2|\lambda|)\|f_1\|_{H^2((0,T)\times\partial\Omega)}\|f_2\|_{H^2((0,T)\times\partial\Omega)}\\
    \leq& C(3+2k)\left(\|\phi_1\|_{H^3(\Omega)}+\left\|\nabla^{-1}\left[\frac{1}{\lambda}\nabla\phi_1\right]\right\|_{H^4(\Omega)}\right)\left(\|\phi_2\|_{H^3(\Omega)}+\left\|\nabla^{-1}\left[\frac{1}{\lambda}\nabla\phi_2\right]\right\|_{H^4(\Omega)}\right)\\
    \leq& C(3+2k)\left(\|\phi_1\|_{H^{\max(s,4)}(\Omega)}+\left\|\nabla^{-1}\left[\frac{1}{\lambda}\nabla\phi_1\right]\right\|_{H^{\max(s,4)}(\Omega)}\right)\\
    &\times\left(\|\phi_2\|_{H^{\max(s,4)}(\Omega)}+\left\|\nabla^{-1}\left[\frac{1}{\lambda}\nabla\phi_2\right]\right\|_{H^{\max(s,4)}(\Omega)}\right)\\
    \leq& C(3+2k)\left((R^2+k^2)^\frac{\max(s,4)}{2}e^{\frac{R}{\sqrt{2}}}+\frac{(R^2+k^2)^\frac{\max(s,4)}{2}}{k}e^{\frac{R}{\sqrt{2}}}+C_q\right)^2\\
    \leq& C(3+2k)\left(\frac{(1+k)^2}{k^2}(R^2+k^2)^{\max(s,4)}e^{\sqrt{2}R}+C_q^2\right)\\
    \leq& C(3+2k)\frac{(1+k)^2}{k^2}\left((R^2+k^2)^{\max(s,4)}+C_q^2\right)e^{\sqrt{2}R}\\
    \leq& C\frac{(3+2k)(1+k)^2}{k^2}\left((R^2+k^2)^{\max(s,4)}+C_q^2\right)e^{\sqrt{2}R}
    \end{aligned}    
$$
Here, the first inequality follows from Remark~\ref{thm:antideri} (with $\Phi = \phi_j$ and $\Psi = -\frac{1}{\lambda}\nabla\phi_j$ to fulfill~\eqref{eq:qcond}), and the third inequality follows from Lemma~\ref{thm:CGOHs}.

As for any $0\neq \xi \in\mathbb{R}^n$ ($n\geq 3$), we have
$$
\hat{\dot{\sigma}}(\xi) = \int_\Omega \dot\sigma e^{- i \xi \cdot x} \,\dif x = \int_{\Omega}\dot\sigma\phi_1\phi_2\dif x -\int_\Omega\dot{\sigma}e^{-i\xi\cdot x}(r_1+r_2+r_1r_2)\dif x.
$$
Thus, the Fourier transform $\hat{\dot\sigma}$ can be estimated as follows:
$$
    \begin{aligned}
    |\hat{\dot{\sigma}}(\xi)|\leq& \left|\int_{\Omega}\dot\sigma\phi_1\phi_2\dif x\right|+\left|\int_\Omega\dot{\sigma}e^{-i\xi\cdot x}(r_1+r_2+r_1r_2)\dif x\right|\\
    \leq& C_\lambda\delta + \|\dot{\sigma}\|_{H^{-s}(\Omega)}\|r_1+r_2+r_1r_2\|_{H^s(\Omega)}\\
    \leq& C_\lambda\delta + \|\dot{\sigma}\|_{H^{-s}(\Omega)}(\|r_1\|_{H^s(\Omega)}+\|r_2\|_{H^s(\Omega)}+\|r_1\|_{H^s(\Omega)}\|r_2\|_{H^s(\Omega)})\\
    \leq& C_\lambda\delta + \|\dot{\sigma}\|_{H^{-s}(\Omega)}\left(\|r_1\|_{H^s(\Omega)}+\|r_2\|_{H^s(\Omega)}+\frac{C_1}{C_0}\|r_2\|_{H^s(\Omega)}\right)\\
    \leq & C_\lambda\delta + \underbrace{ \left[ C_1 \|\sigma_0\|_{H^s(\Omega)}\left(2+\frac{C_1}{C_0}\right)\right] }_{:=C_2}\frac{k}{\sqrt{R^2+k^2}}\|\dot{\sigma}\|_{H^{-s}(\Omega)}
    \end{aligned}
$$
where the last inequality follows from~\eqref{eq:restimate}.

Finally, for an arbitrary number $R_0>0$, we choose $R=R_0$ when $|\xi|\leq R_0+k$ and $R=|\xi|$ when $|\xi|>R_0+k$. This gives the desired estimates.

\end{proof}

Using the pointwise estimate in Lemma~\ref{thm:pointwiseest}, we can obtain the following stability estimate, which shows that the reconstruction of $\dot\sigma$ becomes more stable when $k$ increases, see Remark~\ref{rmk:increasingstability}.

\begin{thm} \label{thm:IncreasingStability}
Let $s>\frac{n}{2}$ ($n\geq 3$). Suppose
$\|\sigma_0\|_{H^s(\Omega)} < C_1 (2+\frac{C_1}{C_0})$ so that $C_2<1$ ($C_2$ is the constant introduced in Lemma~\ref{thm:pointwiseest}) and suppose $\|\dot{\sigma}\|_{H^s(\Omega)}\leq M$ for some $M>0$. Then there exists a constant $C$, independent of $k$ and $\delta$, such that
$$
\|\dot{\sigma}\|_{L^{\infty}(\Omega)} \leq C\left[ \frac{(3+2k)(1+k)^2}{k^2}\left((R_0+k)^{2\max(s,4)}+C_q^2\right) e^{\sqrt{2}R_0}\delta + \left(k+\ln\frac{1}{\delta}\right)^{\frac{n-2s}{2}} \right]^\frac{2s-n}{8s}.
$$

for $0<\delta\leq e^{-1}$, $k\geq k_0$, where $e=2.71828\dots$ is the Euler's constant and $k_0$ is a positive constant. 
\end{thm}
\begin{proof}
By choosing constant $\xi_0 > R_0+k$, we write
$$
    \begin{aligned}
    \|\dot{\sigma}\|_{H^{-s}(\Omega)}^2=&\int_{\mathbb{R}^n}(1+|\xi|^2)^{-s}|\hat{\dot{\sigma}}(\xi)|^2\dif\xi\\
    =&\int_{|\xi|>\xi_0}(1+|\xi|^2)^{-s}|\hat{\dot{\sigma}}(\xi)|^2\dif\xi+\int_{R_0+k<|\xi|\leq\xi_0}(1+|\xi|^2)^{-s}|\hat{\dot{\sigma}}(\xi)|^2\dif\xi\\
    &+\int_{|\xi|\leq R_0+k}(1+|\xi|^2)^{-s}|\hat{\dot{\sigma}}(\xi)|^2\dif\xi\\
    \eqqcolon& I_1+I_2+I_3.
    \end{aligned}
$$
For $I_1$, the H\"older's inequality implies $|\hat{\dot\sigma}(\xi)|\leq C\|\dot\sigma\|_{L^2(\Omega)}$, hence
\begin{equation} \label{eq:I1bound}
    \begin{aligned}
    I_1 & := \int_{|\xi|>\xi_0}(1+|\xi|^2)^{-s}|\hat{\dot{\sigma}}(\xi)|^2\dif\xi \leq C\|\dot\sigma\|_{L^2(\Omega)}^2 \int_{|\xi|>\xi_0}\frac{1}{(1+|\xi|^2)^s}\dif\xi \\
    & \leq C\|\dot\sigma\|_{H^s(\Omega)}^2 \int_{|\xi|>\xi_0}\frac{1}{|\xi|^{2s}}\dif\xi 
    \leq C\|\dot\sigma\|_{H^s(\Omega)}^2 \xi_0^{n-2s}.
    \end{aligned}
\end{equation}
For $I_2$, we denote by $B(0,r)$ the ball of radius $r$ and apply the pointwise estimate in Lemma~\ref{thm:pointwiseest} with $|\xi|>R_0+k$ to get
\begin{equation} \label{eq:I2bound}
    \begin{aligned}
    I_2 & := \int_{R_0+k<|\xi|\leq\xi_0}(1+|\xi|^2)^{-s}|\hat{\dot{\sigma}}(\xi)|^2\dif\xi \\ 
    & \leq \|\hat{\dot{\sigma}}\|_{L^\infty(B(0,\xi_0)\setminus B(0,R_0+k))}^2\int_{\mathbb{R}^n}\frac{1}{(1+|\xi|^2)^s}\dif\xi\\
    & \leq 2C^2\frac{(3+2k)^2(1+k)^4}{k^4}\left((\xi_0^2+k^2)^{\max(s,4)}+C_q^2\right)^2e^{2\sqrt{2}\xi_0}\delta^2 + 2C_2^2\|\dot{\sigma}\|_{H^{-s}(\Omega)}^2,  
    \end{aligned}
\end{equation}
For $I_3$, we apply the pointwise estimate in Lemma~\ref{thm:pointwiseest} with $|\xi|\leq R_0+k$ to get
\begin{equation} \label{eq:I3bound}
    \begin{aligned}
        I_3 := & \int_{|\xi|\leq R_0+k}(1+|\xi|^2)^{-s}|\hat{\dot{\sigma}}(\xi)|^2\dif\xi \\
        \leq&\|\hat{\dot{\sigma}}\|_{L^\infty(B(0,R_0+k))}^2\int_{\mathbb{R}^n}\frac{1}{(1+|\xi|^2)^s}\dif\xi\\
        \leq& 2C^2\frac{(3+2k)^2(1+k)^4}{k^4}\left((R_0^2+k^2)^{\max(s,4)}+C_q^2\right)^2e^{2\sqrt{2}R_0}\delta^2 + 2C_2^2\|\dot{\sigma}\|_{H^{-s}(\Omega)}^2,       
    \end{aligned}
\end{equation}
Henceforth, the estimate will be split into two cases: $R_0+k\leq\frac{1}{2\sqrt{2}}\ln\frac{1}{\delta}$ and $R_0+k > \frac{1}{2\sqrt{2}}\ln\frac{1}{\delta}$.

When $R_0+k\leq\frac{1}{2\sqrt{2}}\ln\frac{1}{\delta}$, we choose $\xi_0=\frac{1}{2\sqrt{2}}\ln\frac{1}{\delta}$. 
Combining the bounds for $I_1, I_2, I_3$ in~\eqref{eq:I1bound} \eqref{eq:I2bound} \eqref{eq:I3bound}, we see that
$$
    \begin{aligned}
    \|\dot{\sigma}\|_{H^{-s}(\Omega)}^2\leq& 2C^2\frac{(3+2k)^2(1+k)^4}{k^4}\left((\xi_0^2+k^2)^{\max(s,4)}+C_q^2\right)^2e^{2\sqrt{2}\xi_0}\delta^2 \\
    &+2C^2\frac{(3+2k)^2(1+k)^4}{k^4}\left((R_0^2+k^2)^{\max(s,4)}+C_q^2\right)^2e^{2\sqrt{2}R_0}\delta^2\\
    &+ C\|\dot\sigma\|_{H^s(\Omega)}^2\xi_0^{n-2s} + 4C_2^2\|\dot{\sigma}\|_{H^{-s}(\Omega)}^2,       
    \end{aligned}
$$
As the assumption ensures $C_2<1$, the $H^{-s}$-norm on the right hand side can be absorbed by the left hand side. We get
$$
\begin{aligned}
    \|\dot{\sigma}\|_{H^{-s}(\Omega)}^2\leq& 2C^2\frac{(3+2k)^2(1+k)^4}{k^4}\left((R_0+k)^{2\max(s,4)}+C_q^2\right)^2e^{2\sqrt{2}R_0}\delta^2\\
    &+2C^2\frac{(3+2k)^2(1+k)^4}{k^4}\left((\xi_0+k)^{2\max(s,4)}+C_q^2\right)^2e^{2\sqrt{2}\xi_0}\delta^2+ CM^2 \xi_0^{n-2s}  \\
    := & 2C^2\frac{(3+2k)^2(1+k)^4}{k^4}\left((R_0+k)^{2\max(s,4)}+C_q^2\right)^2e^{2\sqrt{2}R_0}\delta^2 + I_4 + I_5
    \end{aligned}
$$
In the following, let us use $\lesssim$ to denote an inequality up to a constant factor that is independent of $k$ and $\delta$.
For $I_5$, we have
$$
    \begin{aligned}
        I_5 := & CM^2 \xi_0^{n-2s} \lesssim \left(\ln\frac{1}{\delta}\right)^{n-2s} \\
        \lesssim & \left(\frac{\ln\frac{1}{\delta}}{k+\ln\frac{1}{\delta}}\right)^{n-2s} \left(k+\ln\frac{1}{\delta}\right)^{n-2s} \\
        \lesssim & \left(\frac{2\sqrt{2} (R_0+k)}{k+2\sqrt{2} (R_0+k)}\right)^{n-2s}\left(k+\ln\frac{1}{\delta}\right)^{n-2s} \\
        \lesssim & \left(\frac{2\sqrt{2}}{1+2\sqrt{2}}\right)^{n-2s}\left(k+\ln\frac{1}{\delta}\right)^{n-2s} \\
        \lesssim & \left(k+\ln\frac{1}{\delta}\right)^{n-2s}
    \end{aligned}
$$
where the third line holds because the function $(\frac{t}{k+t})^{n-2s}$ is decreasing in $t>0$ and $\ln\frac{1}{\delta}\geq 2\sqrt{2}(R_0+k)$.
For $I_4$, we use $R_0\leq \xi_0$, $k\leq \xi_0$, and $e^{2\sqrt{2}\xi_0}\delta = 1$ to get
$$
    \begin{aligned}
        \xi_0^{2s-n} I_4 := &  2C^2 (3+2k)^2 \frac{(1+k)^4}{k^4} \xi_0^{2s-n} \left((\xi_0+k)^{2\max(s,4)}+C_q^2\right)^2  e^{2\sqrt{2}\xi_0}\delta^2 \\
        \lesssim &  (1+\xi_0)^2 \frac{(1+k)^4}{k^4} \xi_0^{2s-n} \left(\xi_0^{2\max(s,4)}+C_q^2\right)^2 \delta,
    \end{aligned}
$$
which is bounded since $\lim_{\delta\to 0_+} \delta \xi^N_0 = \lim_{\delta\to 0_+} \delta [\frac{1}{2\sqrt 2}\ln(\frac{1}{\delta})]^N = 0$ for any $N\geq 0$ and $\lim_{k\rightarrow\infty} \frac{(1+k)^4}{k^4} = 1$. Hence,
$$
    I_4 \lesssim \xi^{n-2s}_0 \lesssim I_5 \lesssim \left(k+\ln\frac{1}{\delta}\right)^{n-2s}.
$$
Combining the estimates for $I_4$ and $I_5$, we see that
$$
    \|\dot{\sigma}\|_{H^{-s}(\Omega)}^2 \lesssim \frac{(3+2k)^2(1+k)^4}{k^4}\left((R_0+k)^{2\max(s,4)}+C_q^2\right)^2e^{2\sqrt{2}R_0}\delta^2 + \left(k+\ln\frac{1}{\delta}\right)^{n-2s}
$$
Therefore,
$$
    \|\dot{\sigma}\|_{H^{-s}(\Omega)} \lesssim \frac{(3+2k)(1+k)^2}{k^2}\left((R_0+k)^{2\max(s,4)}+C_q^2\right) e^{\sqrt{2}R_0}\delta + \left(k+\ln\frac{1}{\delta}\right)^{\frac{n-2s}{2}}.
$$

On the other hand, when $R_0+k > \frac{1}{2\sqrt{2}}\ln\frac{1}{\delta}$, we choose $\xi_0=R_0+k$, then $I_2=0$. Combining the bounds for $I_1, I_3$ in~\eqref{eq:I1bound} \eqref{eq:I3bound}, we see that
$$
    \begin{aligned}
    \|\dot{\sigma}\|_{H^{-s}(\Omega)}^2\leq & 2C^2\frac{(3+2k)^2(1+k)^4}{k^4}\left((R_0^2+k^2)^{\max(s,4)}+C_q^2\right)^2e^{2\sqrt{2}R_0}\delta^2\\
    &+ C\|\dot\sigma\|_{H^s(\Omega)}^2\xi_0^{n-2s} + 2C_2^2\|\dot{\sigma}\|_{H^{-s}(\Omega)}^2,       
    \end{aligned}
$$
As $C_2<\frac{1}{2}$, the $H^{-s}$-norm on the right hand side can be absorbed by the left hand side. We get
$$
    \begin{aligned}
    \|\dot{\sigma}\|_{H^{-s}(\Omega)}^2\leq& 2C^2\frac{(3+2k)^2(1+k)^4}{k^4}\left((R_0+k)^{2\max(s,4)}+C_q^2\right)^2e^{2\sqrt{2}R_0}\delta^2 + CM^2 \xi_0^{n-2s}  \\
    := & 2C^2\frac{(3+2k)^2(1+k)^4}{k^4}\left((R_0+k)^{2\max(s,4)}+C_q^2\right)^2e^{2\sqrt{2}R_0}\delta^2 + I_5
    \end{aligned}
$$
This time, we estimate $I_5$ as follows:
$$
    \begin{aligned}
        I_5 := & CM^2 \xi_0^{n-2s} = CM^2 (R_0+k)^{n-2s} \\
        \lesssim & \left(\frac{R_0+k}{k+\ln\frac{1}{\delta}}\right)^{n-2s} \left(k+\ln\frac{1}{\delta}\right)^{n-2s} \\
        \lesssim & \left(\frac{R_0+k}{k+2\sqrt{2} (R_0+k)}\right)^{n-2s}\left(k+\ln\frac{1}{\delta}\right)^{n-2s} \\
        \lesssim & \left(\frac{1}{1+2\sqrt{2}}\right)^{n-2s}\left(k+\ln\frac{1}{\delta}\right)^{n-2s} \\
        \lesssim & \left(k+\ln\frac{1}{\delta}\right)^{n-2s}
    \end{aligned}
$$
where the third line holds because the function $(\frac{R_0+k}{k+t})^{n-2s}$ is increasing in $t>0$ and $\ln\frac{1}{\delta} < 2\sqrt{2}(R_0+k)$.
From this estimate for $I_5$, we see that
$$
    \|\dot{\sigma}\|_{H^{-s}(\Omega)}^2 \lesssim \frac{(3+2k)^2(1+k)^4}{k^4}\left((R_0+k)^{2\max(s,4)}+C_q^2\right)^2e^{2\sqrt{2}R_0}\delta^2 + \left(k+\ln\frac{1}{\delta}\right)^{n-2s}
$$
Therefore,
$$
\|\dot{\sigma}\|_{H^{-s}(\Omega)} \lesssim \frac{(3+2k)(1+k)^2}{k^2}\left((R_0+k)^{2\max(s,4)}+C_q^2\right) e^{\sqrt{2}R_0}\delta + \left(k+\ln\frac{1}{\delta}\right)^{\frac{n-2s}{2}}.
$$

Finally, we interpolate to obtain an estimate for the infinity norm.
Let $\eta>0$ be such that $s=\frac{n}{2}+2\eta$. Choosing $\ell_0=-s$, $\ell_1=s$, $\ell=\frac{n}{2}+\eta=s-\eta$. Then
\[\ell=(1-\tau)\ell_0 + \tau \ell_1,\textup{ where } \tau=\frac{2s-\eta}{2s} \quad \in (0,1).\]
Using the interpolation between Sobolev spaces and the Sobolev embedding, we obtain
$$
\begin{aligned}
\|\dot{\sigma}\|_{L^\infty(\Omega)} \leq& C\|\dot{\sigma}\|_{H^\ell(\Omega)}\leq C\|\dot{\sigma}\|_{H^{-s}(\Omega)}^{1-\tau}\|\dot{\sigma}\|_{H^s(\Omega)}^\tau\leq C M^\tau \|\dot{\sigma}\|_{H^{-s}(\Omega)}^{\frac{2s-n}{8s}}\\
\leq& C M^\tau \left[ \frac{(3+2k)(1+k)^2}{k^2}\left((R_0+k)^{2\max(s,4)}+C_q^2\right) e^{\sqrt{2}R_0}\delta + \left(k+\ln\frac{1}{\delta}\right)^{\frac{n-2s}{2}} \right]^\frac{2s-n}{8s}.
\end{aligned}
$$
\end{proof}

\begin{remark} \label{rmk:increasingstability}
Recall that $k\geq 0$ is an arbitrary non-negative constant. The estimate in Theorem~\ref{thm:IncreasingStability} shows that the stability for $\dot\sigma$ improves as $k$ increases. Indeed, for any fixed $\delta>0$, it is clear that $\left(k+\ln\frac{1}{\delta}\right)^{\frac{n-2s}{2}} \rightarrow 0$ as $k \rightarrow\infty$ since $n-2s<0$. Therefore, the first term in the square parenthesis dominates the right hand side as $k$ increases, yielding a nearly H\"older-type stability estimate.
\end{remark}

\section{Numerical Implementation} \label{sec:numvalid}
This section is devoted to numerical implementation of the reconstruction formula~\eqref{eq:reconstruction1d} in one dimension (1D) when $\rho_0=1$, $\sigma_0=0$. 

The reconstruction procedure in Section~\ref{sec:ConstBGDamp} can be summarized as follows.
\begin{enumerate}
    \item Choose $\lambda = ik$ with $k>0$ and $\theta\in\mathbb{S}^{n-1}$.
    \item Solve the boundary control equations~\eqref{eq:helmholtzsol} for $f$ and $h$.

    \item Compute $\hat{\dot \sigma}(2k\theta)$ from~\eqref{eq:reconstruction1d}.
    \item Repeat the above steps with various $k>0$ and $\theta\in\mathbb{S}^{n-1}$ to recover the Fourier transform $\hat{\dot \sigma}$.
    \item Invert the Fourier transform to recover $\dot\sigma$.
\end{enumerate}
Step 2 requires solving boundary control equations, for which the existence of solutions are ensured by Proposition~\ref{prop:control}. In the special case that $\rho_0\equiv 1$ and $\sigma_0\equiv 0$, we can adapt the idea in the authors' earlier work~\cite{oksanen2022linearized} to give analytic solutions using a simple time reversal procedure.

\subsection{Computing Boundary Controls using Time Reversal}

Let $\Omega=(a,b)$ be a 1D interval with end points $a,b$ ($a<b$). Given a smooth function $\phi\in C^{\infty}([a,b]])$, we can extend it from $[a,b]$ to $\mathbb{R}$ as follows:
$$
    \tilde\phi :=
    \begin{cases}
    \phi&x\in[a,b],\\
    \phi\cdot\exp\{1-\frac{1}{1-(x-a)^{2d}}\}&x\in(a-1,a),\\
    \phi\cdot\exp\{1-\frac{1}{1-(x-b)^{2d}}\}&x\in(b,b+1),\\
    0&x\notin(a-1,b+1),
    \end{cases}
$$
where $d$ is a positive integer. It is easy to verify that $\tilde{\phi}$ is $C^{2d-1}$ at $x=a,b$ and $C^\infty$ at other points. To guarantee the existence of the second-order derivative of Neumann data, we take $d\geq2$.

For any two smooth functions $\phi,\psi\in C^\infty([a,b])$ with extensions $\tilde{\phi}, \tilde{\psi}$, we consider the 1D initial value problem for the wave equation:
$$
\left\{
\begin{array}{rlll}
    \partial_t^2w(t,x)-\Delta w(t,x)&=0\quad&&\text{ in }(0,2T)\times\mathbb{R},\\
    w(T,x)&=\tilde\phi(x)+\frac{1}{2}\int_{-\infty}^{\infty}\tilde{\psi}(\xi)\dif\xi\qquad&&\text{ on } \mathbb{R},\\
    \partial_t w(T,x)&=\tilde\psi(x)&&\text{ on } \mathbb{R}.\\
\end{array}
\right.
$$
Its solution, according to the D'Alembert formula, is
\begin{equation} \label{eq:dalembert}
    w(t,x) = \frac{1}{2}\left[\tilde{\phi}(x+T-t)+\tilde{\phi}(x-T+t)\right]-\frac{1}{2}\int_{x-T+t}^{x+T-t}\tilde{\psi}(\xi)\dif\xi+\frac{1}{2}\int_{-\infty}^{\infty}\tilde{\psi}(\tau)\dif\tau.
\end{equation}
Moreover, it can be analytically verified that $w(0,x) = 0$ and $\partial_t w(0,x) = 0$ in $\Omega$ (this is why the constant $C_q=\frac{1}{2}\int^\infty_{-\infty} \tilde{\psi}(\xi)\,\dif\xi$ is chosen in $w(T,x)$).
We conclude that $w|_{(0,2T)\times\Omega}$ is a solution of the boundary value problem~\eqref{eq:wave0} (with $\rho_0 \equiv 1$ and $\sigma_0\equiv 0$). 
As a result, the boundary control equation $w^f(T)=\phi$ in $\Omega$ admits the explicit solution 
\begin{equation}\label{eq:reversal}
f(t,x) = \partial_\nu w(t,x) = \pm\frac{1}{2}\left.\left[\tilde{\phi}'(x+T-t)+\tilde{\phi}'(x-T+t)+\tilde{\psi}(x-T+t)-\tilde{\psi}(x+T-t)\right]\right|_{x=a,b}
\end{equation}
where we take $+$ at $x=b$ and $-$ at $x=a$. Higher order derivatives such as $\partial_t f$ and $\partial^2_t f$ can be computed based on this explicit form. These derivatives are needed in order to compute the Fourier transform $\hat{\dot{\sigma}}$, see~\eqref{eq:reconstruction1d}.

If we introduce $p^f=\partial_t w$ and $q^f=\nabla w$.
Given any $\phi\in C^\infty([a,b])$, the analysis above shows that the boundary control equation $p^f(T)=\phi$ in $\Omega$ can be solved as follows:
\begin{enumerate}
    \item Define $\psi:=-\frac{1}{\lambda} \nabla \phi$ to fulfill~\eqref{eq:final}.
    \item Compute $w$ using~\eqref{eq:dalembert}.
    \item Compute $f$ using~\eqref{eq:reversal}. 
\end{enumerate}

This construction offers an advantage as the boundary control equation allows for explicit solutions. The rest of the construction goes as follows.
Choose $\lambda = ik$. For a fixed $k>0$, define
\[\phi_k(x) :=\frac{i}{k}\exp(ikx)+\frac{1}{2}\int_{-\infty}^{\infty}\tilde{\psi}_k(\xi)\dif\xi,\qquad\psi_k(x)=\exp(ikx).\]
Solve the boundary control equation $u^{f_k}_0(T)=\phi_k$ analytically to find $f_k$. Then
\begin{align*}
q_0^{f_k}(T) & = \nabla u_0^{f_k}(T) = \nabla\phi_k(x) =-\exp(ikx) \\
p_0^{f_k}(T) & =\partial_t u_0^{f_k}(T)=\exp(ikx).
\end{align*}
Finally, take two different $k_1,k_2>0$ and insert the analytic construction of $f_{k_1}, f_{k_2}, p^{f_{k_1}}_0(T), p^{f_{k_2}}_0(T)$ into~\eqref{eq:control}. We obtain the Fourier transform
\[\hat{\dot{\sigma}}(k_1+k_2)=\langle p_0^{f_{k_1}}(T),p_0^{f_{k_2}}(T)\rangle_{L^2(\Omega,\dot\sigma\dif x)}.\]

\subsection{Numerical experiments}
We take the 1D computational domain $\Omega=[-1,1]$ and $T=5$. The wave equation is solved using the second order central difference scheme with spacing $\Delta x = \frac{1}{250}$ and $\Delta t = \frac{1}{2500}$. We choose $p_0(T)$ from the following Fourier basis (which are the real and imaginary parts of $\exp(ikx)$ with $k = \frac{\pi}{2}, \dots, \frac{N\pi}{2}$): 
\[\left\{ 1, \sin\left(\frac{\pi}{2}x\right),\cos\left(\frac{\pi}{2}x\right),\dots,\sin\left(\frac{N\pi}{2}x\right),\cos\left(\frac{N\pi}{2}x\right) \right\}\]
with $N=10$. Since $q_0(T)$ is related to $p_0(T)$ by the equation \eqref{eq:final}, we take 
\[q_0(T)=-\frac{1}{\lambda}\nabla p_0(T)\]
with $\lambda = ki =\frac{\pi}{2}i,\dots,\frac{N\pi}{2}i$. 
By explicitly constructing the boundary controls using the time reversal method in the previous section and \eqref{eq:reconstruction1d}, we obtain the Fourier modes of $\dot{\sigma}$ using the trigonometric relations 
\[\sin^2\theta+\cos^2\theta=1,\quad\sin2\theta=2\sin\theta\cos\theta,\quad\cos2\theta=\cos^2\theta-\sin^2\theta\]

\begin{algorithm}[!h]
    \SetAlgoLined
	\caption{Reconstruction}  \label{alg:reconstruction}
 \vspace{1ex}
	\KwIn{ND map $\dot{\Lambda}$, background density $\rho_0\equiv 1$, background damping $\sigma_0\equiv 0$, noise level $\varepsilon$ and domain $\Omega$ \medskip}
	
		 Choose the Fourier truncation $N$. 
         
         \For{$k = 1,2,\dots, N$}{
            
          $p_0(T)\gets\sin\left(\frac{k\pi}{2}x\right)$, $q_0(T)\gets\frac{2}{k\pi i}\cos\left(\frac{k\pi}{2}x\right)$.
   
	   Construct the boundary control $f_k$ and its derivatives using the time reversal method, see \eqref{eq:reversal}.

        Calculate $\dot\Lambda f_k$ and its derivatives by numerically solving $\eqref{eq:wave0}$ and $\eqref{eq:wavedot}$.

        $p_0(T)\gets\cos\left(\frac{k\pi}{2}x\right)$, $q_0(T)\gets-\frac{2}{k\pi i}\sin\left(\frac{k\pi}{2}x\right)$.

        Construct boundary control $h_k$ and its derivatives using time reversal method.

        Calculate $\dot\Lambda h_k$ and its derivatives.

        Add Gaussian random noise with level $\varepsilon$ to $\dot\Lambda f_k$, $\dot\Lambda h_k$ and their derivatives independently.

        Calculate $a_k\coloneqq\langle p_0^{h_k},p_0^{h_k}\rangle_{L^2(\Omega,\dot\sigma\dif x)}-\langle p_0^{f_k},f_0^{f_k}\rangle_{L^2(\Omega,\dot\sigma\dif x)}$, $b_k\coloneqq 2\langle p_0^{f_k},p_0^{h_k}\rangle_{L^2(\Omega,\dot\sigma\dif x)}$ using \eqref{eq:reconstruction1d}.

           }
    
        Calculate $a_0\coloneqq\langle p_0^{h_1},p_0^{h_1}\rangle_{L^2(\Omega,\dot\sigma\dif x)}+\langle p_0^{f_1},f_0^{f_1}\rangle_{L^2(\Omega,\dot\sigma\dif x)}$.

        The reconstructed $\tilde{\dot{\sigma}}$ is given by $\frac{a_0}{2}+\sum_{k=1}^Na_k\sin(k\pi x) + b_k\cos(k\pi x)$.  \medskip

        \KwOut{sound speed perturbation $\dot{\rho}$}
\end{algorithm}

\textbf{Experiment 1:}
In this experiment, we start with the ground-truth perturbation 
\[\dot{\sigma} = \cos(\pi x) + \cos(2\pi x) + \cos(3\pi x) + \sin(4\pi x)+4,\]
which can be represented using the Fourier basis. The graph of $\dot\sigma$ is shown in Figure~\ref{fig:groundtruth1}. We add $0\%,1\%,5\%$ Gaussian noise to the measurement $\dot\Lambda$, respectively. The reconstructions and corresponding errors are illustrated in Figure~\ref{fig:reconstruction1}.
\begin{figure}[!h]
    \centering
    \includegraphics[width=0.49\textwidth]{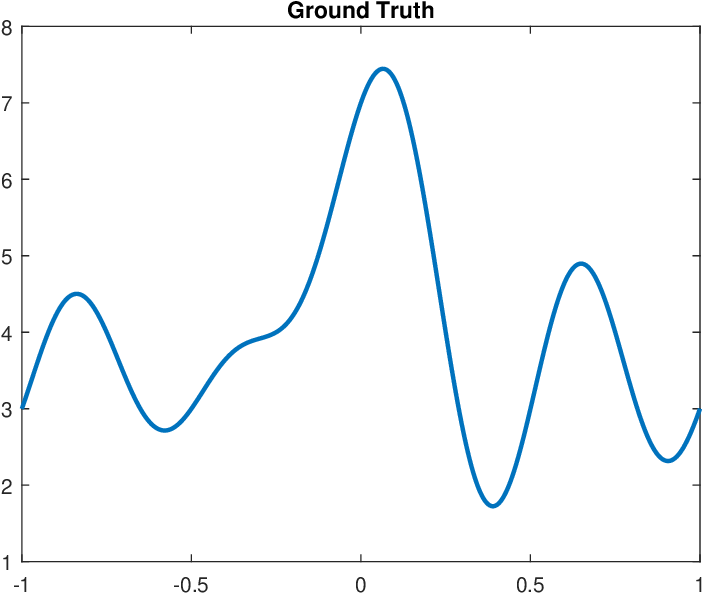}
    \caption{Continuous ground truth $\dot{\sigma} = \cos(\pi x) + \cos(2\pi x) + \cos(3\pi x) + \sin(4\pi x)+4$.}
    \label{fig:groundtruth1}
\end{figure}
\begin{figure}[!h]
    \centering
    \includegraphics[width=0.48\textwidth]{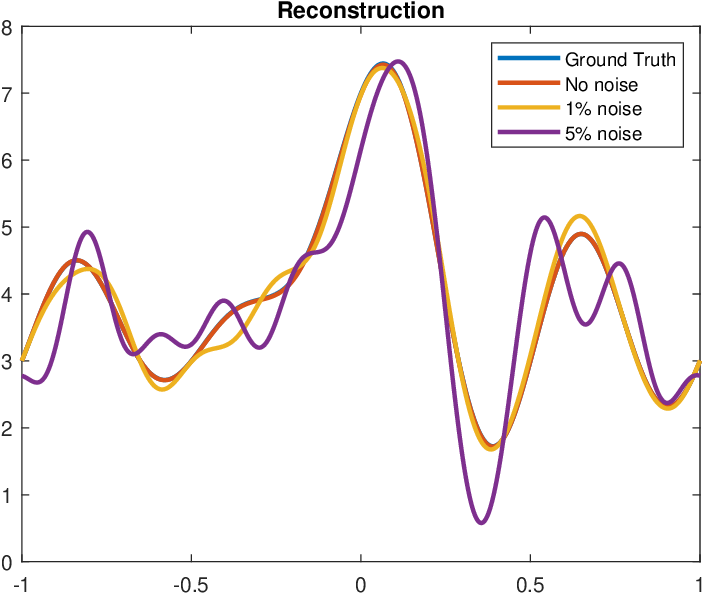}
    \includegraphics[width=0.49\textwidth]{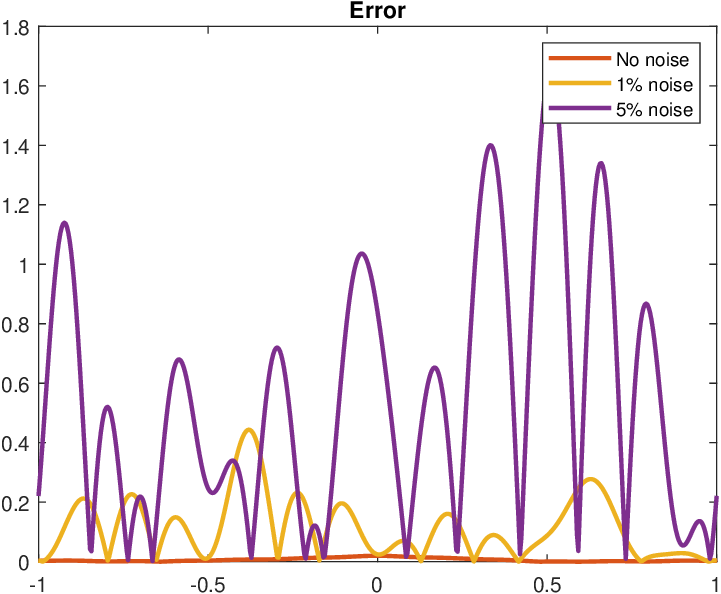}
    \caption{Left: Reconstructed $\dot\sigma$ with $0\%,1\%,5\%$ Gaussian noise and the ground truth. Right: The corresponding error between the reconstruction result and the ground truth. The relative $L^2$-errors are $0.2\%,3.5\%$ and $16.2\%$, respectively.}
    \label{fig:reconstruction1}
\end{figure}

\medskip
\textbf{Experiment 2:}
In this experiment, we consider a piecewise discontinuous perturbation
\[\dot{\sigma}=\left\{
    \begin{alignedat}{2}
        2&&\qquad-1\leq x&\leq-\frac{1}{2}\\
        \frac{3}{2}&&-\frac{1}{2}<x&<\hphantom{-}\frac{1}{3}\\
        1&&\frac{1}{3}\leq x&\leq\hphantom{-}1
    \end{alignedat}
\right.\]
Its Fourier series is
\[\dot{\sigma}=\frac{35}{24}+\sum_{k=1}^\infty\left[\frac{\sin\left(\frac{k\pi}{3}\right)-\sin\left(\frac{k\pi}{2}\right)}{2k\pi}\cos(k\pi x)-\frac{\cos\left(\frac{k\pi}{3}\right)+\cos\left(\frac{k\pi}{2}\right)-2\cos(k\pi)}{2k\pi}\sin(k\pi x)\right].\]
With the choice of the truncated Fourier basis, we can only expect to reconstruct the orthogonal projection:
\[\dot{\sigma}_N\coloneqq\frac{35}{24}+\sum_{k=1}^N\left[\frac{\sin\left(\frac{k\pi}{3}\right)-\sin\left(\frac{k\pi}{2}\right)}{2k\pi}\cos(k\pi x)-\frac{\cos\left(\frac{k\pi}{3}\right)+\cos\left(\frac{k\pi}{2}\right)-2\cos(k\pi)}{2k\pi}\sin(k\pi x)\right],\]
see Figure~\ref{fig:groundtruth2}. We plot the reconstruction and the corresponding error with respect to the orthogonal projection $\dot{\sigma}_N$ in Figure~\ref{fig:reconstruction2}.

\begin{figure}[!h]
    \centering
    \includegraphics[width=0.49\textwidth]{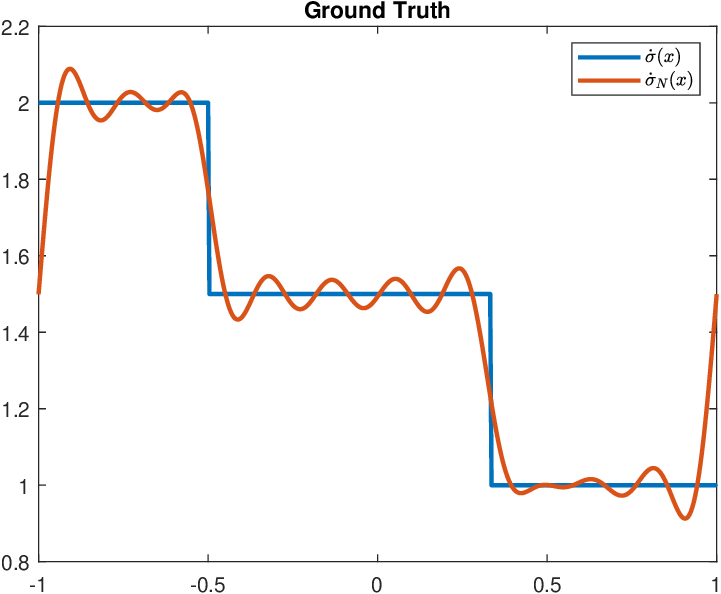}
    \caption{Piecewise-constant ground truth $\dot\sigma$ and its Fourier truncation $\dot\sigma_N$ with $N=10$.}
    \label{fig:groundtruth2}
\end{figure}
\begin{figure}[!h]
    \centering
    \includegraphics[width=0.49\textwidth]{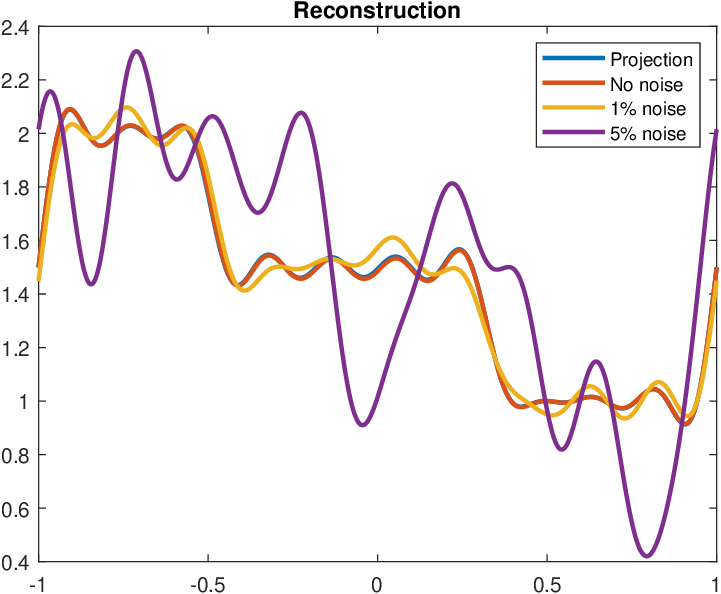}
    \includegraphics[width=0.49\textwidth]{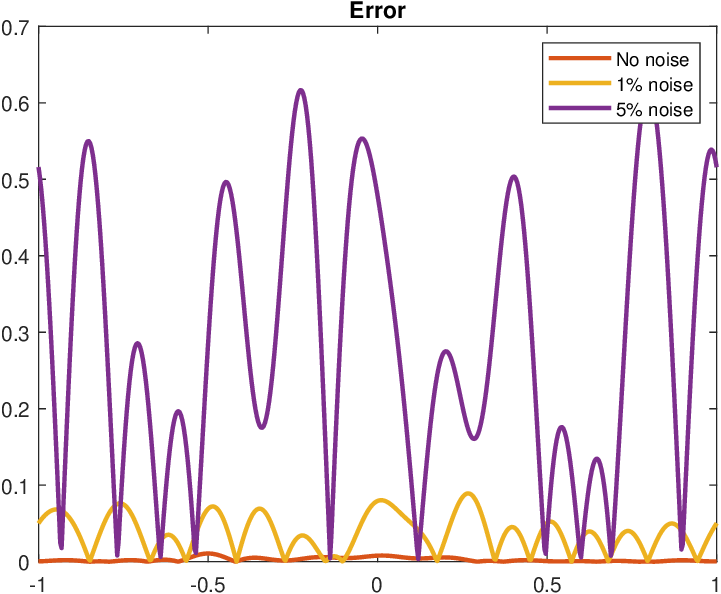}
    \caption{Left: The reconstructions with $0\%,1\%,5\%$ Gaussian noise and the orthogonal projection $\dot\sigma_N$. Right: Errors between the reconstructions and the orthogonal projection $\dot\sigma_N$. The relative $L^2$-errors are $0.2\%,3.0\%$ and $22.5\%$, respectively.}
    \label{fig:reconstruction2}
\end{figure}

\medskip
\textbf{Experiment 3:}
In this experiment, we apply the algorithm in a non-linear case. The absorption coefficient is given by 
\[\sigma = \sigma_0 + \varepsilon\dot{\sigma} + \varepsilon^2\ddot{\sigma},\]
where $\varepsilon>0$ is a small constant and
\[\dot{\sigma} = \cos(\pi x) + \cos(2\pi x) + \cos(3\pi x) + \sin(4\pi x)+4,\qquad\ddot{\sigma} = 200\sin(20\pi x).\]
In this case, since
\[\Lambda_\sigma-\Lambda_{\sigma_0}\approx\varepsilon\dot\Lambda_{\dot\sigma}=\dot\Lambda_{\varepsilon\dot\sigma}\]
when $\varepsilon$ is small, we replace $\dot\Lambda f$ with $\Lambda_\sigma f-\Lambda_{\sigma_0}f$ in \eqref{eq:reconstruction1d}, where $\Lambda_\sigma f$ and $\Lambda_{\sigma_0}f$ are computed by numerically solving \eqref{eq:wave}. Applying Algorithm~\ref{alg:reconstruction} gives us an approximation of $\sigma$. We choose $\varepsilon=10^{-3}$, the ground truth is illustrated in Figure~\ref{fig:groundtruth3}. 

When adding Gaussian noise, we added the noise to the difference $\Lambda_\sigma f-\Lambda_{\sigma_0}f$ rather than adding it to $\Lambda_\sigma f$ and $\Lambda_{\sigma_0}f$ respectively. This is because $|\Lambda_\sigma f|$ and $|\Lambda_{\sigma_0}f|$ could be large when $|\Lambda_\sigma f-\Lambda_{\sigma_0}f|$ is small, adding noise independently could cause larger error than adding noise to the difference with the same noise level. The reconstruction result and the corresponding error with $0\%,1\%,5\%$ Gaussian noise are illustrated in Figure~\ref{fig:reconstruction3}.

\begin{figure}[!h]
    \centering
    \includegraphics[width=0.49\textwidth]{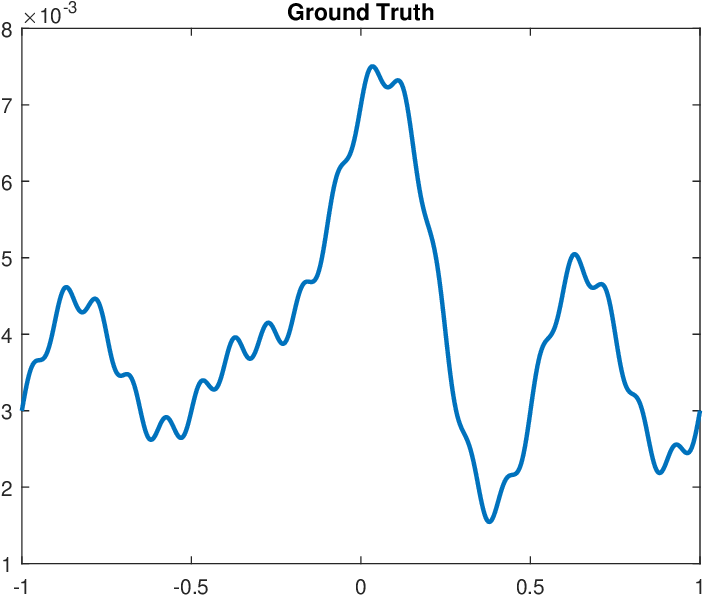}
    \caption{Ground truth $\dot\sigma$}
    \label{fig:groundtruth3}
\end{figure}
\begin{figure}[!h]
    \centering
    \includegraphics[width=0.49\textwidth]{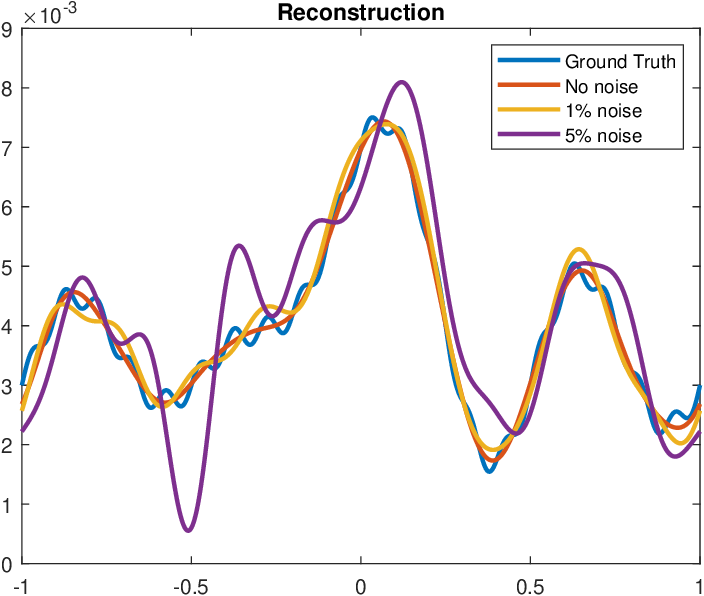}
    \includegraphics[width=0.49\textwidth]{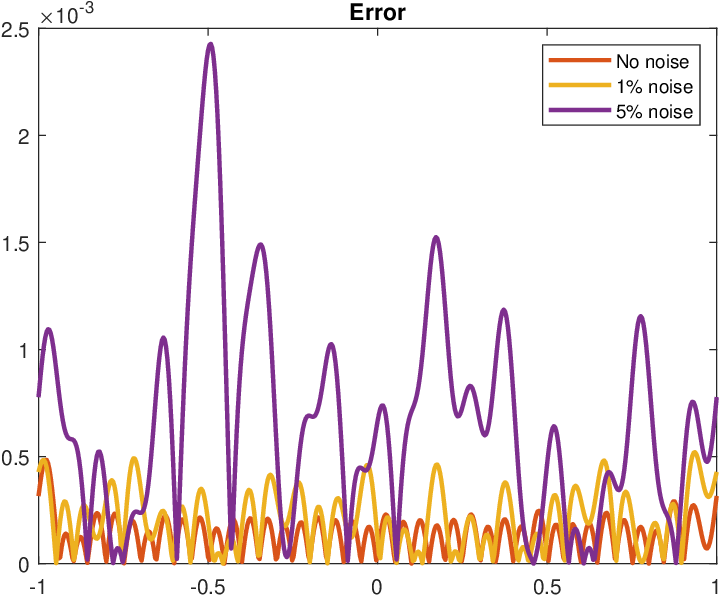}
    \caption{Left: Reconstructed $\dot\sigma$ with $0\%,1\%,5\%$ Gaussian noise and the ground truth. Right: The corresponding error between the reconstruction result and the ground truth. The relative $L^2$-errors are $3.7\%,5.9\%$ and $19.4\%$, respectively.}
    \label{fig:reconstruction3}
\end{figure}

\section*{Acknowledgment}
The research is partially supported by the NSF grants DMS-2237534 and DMS-2220373. The authors thank Dr. Lauri Oksanen for helpful discussion.

\section{Appendix}

In this appendix, we prove the boundedness of the ND map~\eqref{eq:NDmap} between suitable spaces for small $\sigma$.

The following well-posedness result for the undamped wave equation with $\sigma\equiv 0$ is well known (see, e.g., \cite{evans1998partial}):
For $\mathfrak{f}\in H^s((0,2T)\times\Omega)$, the boundary value problem
\begin{equation}
\left\{
\begin{alignedat}{2}
    \square_{\rho,0} v(t,x)&=\mathfrak{f} \quad&&\text{ in }(0,2T)\times\Omega,\\
    v(0,x)=\partial_t v(0,x)&=0\quad&&\text{ on } \Omega,\\
    \partial_\nu v(t,x)&=0\quad&&\text{ on }(0,2T)\times\partial\Omega,
\end{alignedat}
\right.
\end{equation}
admits a unique solution $v\in H^{s+1}((0,2T)\times\Omega)$. Moreover, there exists a constant $C>0$ such that 
$$
\|v\|_{H^{s+1}((0,2T)\times\Omega)} \leq C  \|\mathfrak{f}\|_{H^s((0,2T)\times\Omega)}.
$$
If we introduce the solution operator $S\mathfrak{f}:=v$, then $S: H^s((0,2T)\times\Omega) \rightarrow H^{s+1}((0,2T)\times\Omega)$ is a bounded linear operator.

\begin{prop} \label{thm:Lambdaboundedness}
    Let $s\geq \frac{1}{2}$ be a real number. If  $\|\sigma\|_{W^{s-\frac{1}{2},\infty}(\Omega)}$ is sufficiently small, then for any $f\in H_{00}^s((0,2T)\times\partial\Omega)$, the boundary value problem~\eqref{eq:wave} admits a unique solution $u\in H^{s+\frac{1}{2}}((0,2T)\times\Omega)$. As a result, the Neumann-to-Dirichlet map
 $$
 \Lambda_\sigma:H_{00}^s((0,2T)\times\partial\Omega) \rightarrow H^s((0,2T)\times\partial\Omega)
$$ 
is a bounded linear operator.
\end{prop}

\begin{proof}
The linearity is clear. To show the boundedness, take $f\in C^\infty_c((0,2T]\times\partial\Omega)$ and extend it to $F\in H^{s+\frac{3}{2}}((0,2T)\times\Omega)$ such that $\partial_\nu F=f$ and $F(t,x)=0$ for any $x\in\Omega$ and $t$ close to $0$. Such $F$ can be chosen to satisfy
\begin{equation} \label{eq:fextension}
\|F\|_{H^{s+\frac{3}{2}}((0,2T)\times\Omega)}\leq C\|f\|_{H^s((0,2T)\times\partial\Omega)}.
\end{equation}
Denote $v\coloneqq u^f-F$ where $u^f$ is the solution of~\eqref{eq:wave} with the Neumann boundary condition $f$, then $v$ satisfies
$$
\left\{
\begin{alignedat}{2}
\square_{\rho,0}v(t,x)&=-\square_{\rho,\sigma}F(t,x) - \sigma(x) \partial_t v \quad&&\text{ in }(0,2T)\times\Omega,\\
    v(0,x)=\partial_tv(0,x)&=0\quad&&\text{ on } \Omega,\\
    \partial_\nu v(t,x)&=0\quad&&\text{ on }(0,2T)\times\partial\Omega.
\end{alignedat}
\right.
$$
Applying the solution operator $S$ , we get an integral equation
$$
\left( I +  S \circ \sigma \circ \partial_t  \right) v = -S(\Box_{\rho,\sigma} F)
$$
Note that $\partial_t: H^{s+\frac{1}{2}}((0,2T)\times\Omega) \rightarrow H^{s-\frac{1}{2}}((0,2T)\times\Omega)$ is bounded and multiplication by $\sigma\in C^\infty(\overline{\Omega})$ is bounded on $H^{s-\frac{1}{2}}((0,2T)\times\Omega)$. Therefore, if $\|\sigma\|_{W^{s-\frac{1}{2},\infty}(\Omega)}$ is sufficiently small, the operator $I+S\circ\sigma\circ\partial_t$ is boundedly invertible on $H^{s+\frac{1}{2}}((0,2T)\times\Omega)$. Hence,
$$
v = -(I+S\circ\sigma\circ\partial_t)^{-1} S(\Box_{\rho,\sigma} F) \qquad \in H^{s+\frac{1}{2}}((0,2T)\times\Omega)
$$
with the norm estimate
\begin{equation} \label{eq:vestimate}
\|v\|_{H^{s+\frac{1}{2}}((0,2T)\times\Omega)} 
\leq C \|S(\Box_{\rho,\sigma} F)\|_{H^{s+\frac{1}{2}}((0,2T)\times\Omega)} \leq C \|\Box_{\rho,\sigma} F\|_{H^{s-\frac{1}{2}}((0,2T)\times\Omega)} 
\leq C \|F\|_{H^{s+\frac{3}{2}}((0,2T)\times\Omega)}
\end{equation}
for some constant $C>0$. Combining~\eqref{eq:fextension} \eqref{eq:vestimate} and the continuity of the trace operator, we conclude 
\begin{align*}    
 \|\Lambda_\sigma f \|_{H^s((0,2T)\times\partial\Omega)} & \leq C \|u \|_{H^{s+\frac{1}{2}}((0,2T)\times\Omega)} \\
 & \leq C \left( \|v\|_{H^{s+\frac{1}{2}}((0,2T)\times\Omega)} + \|F\|_{H^{s+\frac{1}{2}}((0,2T)\times\Omega)} \right) \\
 & \leq C \|f\|_{H^s((0,2T)\times\partial\Omega)}.
 \end{align*}
This proves the claim for $f\in C^\infty_c((0,2T]\times\partial\Omega)$. The general case follows from the density of such $f$ in $H_{00}^s((0,2T)\times\partial\Omega)$.
\end{proof}

\bibliographystyle{abbrv}
\bibliography{ref}
\end{document}